\documentclass[10pt,a4paper]{article}
\usepackage[a4paper]{geometry}
\usepackage{authblk}
\usepackage{amssymb,latexsym,amsmath,amsfonts,amsthm}
\usepackage{graphicx}
\usepackage{epsfig}
\usepackage{tikz}
\usepackage[T1]{fontenc}
\usepackage{pdfsync}

\newcommand{\supp}{\mathrm{supp}}

\newcommand{\lf}{\left\lfloor}
\newcommand{\rf}{\right\rfloor}
\newcommand{\lc}{\left\lceil}
\newcommand{\rc}{\right\rceil}
\newcommand{\sign}{\mathrm{sign}}

\newtheorem{theorem}{Theorem}[section]
\newtheorem{lemma}[theorem]{Lemma}
\newtheorem{proposition}[theorem]{Proposition}

\newtheorem{corollary}[theorem]{Corollary}

\theoremstyle{definition}
\newtheorem{definition}[theorem]{Definition}

\theoremstyle{remark}

\numberwithin{equation}{section}

\title{Nikishin systems on star-like sets: algebraic properties and weak asymptotics of the associated multiple orthogonal polynomials}

\date{\today}

\author[1]{Abey L\'{o}pez-Garc\'{i}a}
\author[2]{Erwin Mi\~{n}a-D\'{i}az}
\affil[1]{University of South Alabama, Department of Mathematics and Statistics, ILB 325, 411 University
Blvd North, Mobile AL, 36688.
E-mail address: lopezgarcia@southalabama.edu}
\affil[2]{The University of Mississippi, Department of Mathematics, Hume Hall 305,
University, MS 38677.
E-mail address: minadiaz@olemiss.edu}

\begin{document}

\maketitle

\abstract{Polynomials $Q_n(z)$, $n=0,1,\ldots,$ that are multi-orthogonal with respect to  a Nikishin system of $p\geq 1 $ compactly supported measures over  the star-like set of $p+1$ rays $S_+:=\{z\in \mathbb{C}:  z^{p+1}\geq 0 \}$ are investigated.  We prove that the Nikishin system is normal, that the polynomials satisfy a three-term recurrence relation of order $p+1$ of the form $z Q_{n}(z)=Q_{n+1}(z)+a_{n}\,Q_{n-p}(z)$ with $a_n>0$ for all $n\geq p$, and that the nonzero roots of $Q_n$ are all simple and located in $S_+$. Under the assumption of regularity (in the sense of Stahl and Totik) of the measures generating the Nikishin system, we describe the asymptotic zero distribution and weak behavior of the polynomials $Q_n$ in terms of a vector equilibrium problem for logarithmic potentials. Under the same regularity assumptions, a theorem on the convergence of the Hermite-Pad\'{e} approximants to the Nikishin system of Cauchy transforms is proven.}

\section{Introduction}
This work is motivated by the studies \cite{AptKalIseg,AptKalSaff,DelLop} on sequences of polynomials $\{Q_n\}_{n=0}^\infty$ satisfying a recurrence relation of the form
\begin{equation}\label{generalrecurrence}
z Q_{n}(z)=Q_{n+1}(z)+a_{n}\,Q_{n-p}(z),\quad a_{n}>0,\ n\geq p,
\end{equation}
where $p$ is a fixed positive integer.

Some well-known families of polynomials satisfy this type of recurrence relation with the coefficients $a_n$ all being equal to some constant $a$. For instance, when $p=1$ and $a_n=1$ for all $n\geq 1$, the polynomials $Q_n$ resulting from the pairs of initial conditions $Q_0(z)=2$, $Q_1(z)=z$, and  $Q_0(z)=1$, $Q_1(z)=z$, are, respectively, the Chebyshev polynomials of the first and second kind for the interval $[-2,2]$. As a way of generalizing the Chebyshev polynomials of the first kind, one can  set in \eqref{generalrecurrence} $a_n=1/p$, $n\geq p$, and $Q_0(z)=p+1$, $Q_\ell=z^\ell$, $\ell=1,\ldots,p$, which generates the sequence of Faber polynomials associated with a hypocycloid of $p+1$ cusps. Many interesting properties of these Faber polynomials were established in \cite{HeSaff}. For instance, their zeros are all located in the star-like set of $p+1$ rays
\[
S_{+}:=\{z\in\mathbb{C}: z^{p+1}\geq 0\},
\]
more precisely, they are contained, interlace, and form a dense subset of  $\{z\in S_+:|z|<(p+1)/p\}$.

It was proven in \cite{AptKalIseg} that with the initial conditions
\begin{equation}\label{initialconditions}
Q_\ell(z)=z^\ell,\quad 0\leq \ell\leq p,
\end{equation}
the polynomials generated by \eqref{generalrecurrence} are in fact multi-orthogonal (in the same non-Hermitian sense of Definition \ref{multiorthogonalitydef} below) with respect to a system of $p$ complex measures $\mu_1,\ldots,\mu_{p}$ supported on $S_+$. These measures can be viewed as spectral measures \cite{AptKalIseg,AptKalSaff} of the difference operator given in the standard basis of the Hilbert space $l^2(\mathbb{N})$ by the infinite $(p+2)$-banded Hessenberg matrix
\begin{equation}\label{operator}
\left(
\begin{array}{ccccccc}
0 &1&0&0&0&\ldots &\ldots\\
0 &0&1&0&0&\ldots &\ldots\\
0 &0&0&1&0&\ldots &\ldots\\
\ldots &\ldots&\ldots&\ldots&\ldots&\ldots &\ldots\\
a_p &0&0&0&0&\ldots &\ldots\\
0 &a_{p+1}&0&0&0&\ldots &\ldots\\
0 &0 & a_{p+2}&0&0&\ldots &\ldots\\
\ldots &\ldots&\ldots&\ldots&\ldots&\ldots &\ldots
\end{array}
\right).
\end{equation}
Their Cauchy transforms
\[
\int_{S_+}\frac{d\mu_j(t)}{z-t},\quad 1\leq j\leq p,
\]
are the resolvents or Weyl functions of the operator. We remark that the spectral measures are of the form $d\mu_{j}(t)=t^{1-j} d\nu_{j}(t^{p+1})$, $j=1,\ldots,p,$ where $\nu_{j}$ is a positive measure supported on $S_{+}^{p+1}=\mathbb{R}_{+}$. Hence $\mu_{1}$ is rotationally invariant, and the rest are rotationally invariant up to a monomial factor.

The $l^1$ perturbation of the constant coefficient case
\begin{equation}\label{perturbationcond}
\sum_{n=p}^\infty|a_n-a|<\infty, \quad a>0,
\end{equation}
was investigated in \cite{AptKalSaff}. Here, the strong asymptotics of the polynomials $Q_n$ determined by \eqref{generalrecurrence}, \eqref{initialconditions}, and \eqref{perturbationcond}, as well as properties of the  measures $\mu_j$, were derived. For instance, it was proven that these spectral measures are absolutely continuous, and a formal connection of these measures with a Nikishin-type system was obtained. In \cite{DelLop}, this connection was explicitly established in the case of periodic recurrence coefficients (see Section 2.4 of \cite{DelLop}), and many algebraic and asymptotic properties of the Riemann-Hilbert minors associated with the polynomials $Q_{n}$ satisfying \eqref{generalrecurrence}--\eqref{initialconditions} were given.

Motivated by these results, we investigate in this paper polynomials $Q_n$ that are multi-orthogonal with respect to a Nikishin system of $p$ measures (defined in analogy to the classical sense) supported over the star-like set $S_+$. As we will see in Section \ref{sectionrecurrence} below, such $Q_n$'s happen to satisfy the recurrence relation \eqref{generalrecurrence}. Our goal is to understand how the properties of the measures generating the Nikishin system affect the multi-orthogonal polynomials $Q_n$ and the recurrence coefficients $a_n$, in particular, what their asymptotic behavior is as $n\to \infty$. Thus, in the context of inverse spectral problems, our investigation sheds some light into the properties of the operator \eqref{operator}.

Nikishin systems of functions (the Cauchy transforms of a Nikishin system of measures  on intervals of the real line) were first introduced in \cite{Nik} as the first wide class of functions possessing convergent  Hermite-Pad\'{e} approximants. While in his original paper \cite{Nik} Nikishin proved this convergence only for a system of two measures and diagonal multi-indices, great progress has been made since then for any number of intervals and arbitrary multi-indices (see, for instance, \cite{prietolagomasino}).
Our work can also be viewed within the context of rational approximation as a generalization of Nikishin systems from real intervals to star-like sets.

The content of the paper is organized in five sections. Sections  \ref{sectionnikishinsystems} and \ref{sectionrecurrence} are, for the most part, of an algebraic nature, and they have been linearly structured so as to have any result needed for a given topic stated and proven beforehand. The Nikishin system and other related hierarchies of measures, together with the multi-orthogonal polynomials and their associated  functions of the second kind, are introduced in Section  \ref{sectionnikishinsystems}. Among the many relations and properties proven in that section figure the normality of the Nikishin system and the location of the zeros of the multi-orthogonal polynomials and of the functions of the second kind. In Section \ref{sectionrecurrence}, we prove the recurrence relation \eqref{generalrecurrence} for the multi-orthogonal polynomials, including the (nontrivial) positivity of the recurrence coefficients. In Section 5, we describe the asymptotic zero distribution and weak behavior of the polynomials $Q_n$ in terms of a vector equilibrium problem for logarithmic potentials, under the assumption that the measures generating the Nikishin system are regular in the sense of Stahl and Totik. A weak convergence theorem for the coefficients of the recurrence relation is also obtained. Finally, in Section \ref{sectionhermitepade}, and under the same regularity assumptions, a theorem on the convergence of the Hermite-Pad\'{e} approximants to the Nikishin system of Cauchy transforms is proven.

Many of the results in this paper were already obtained in \cite{LopGar} for Nikishin systems of $p=2$ measures on a star-like set of three rays. The case when $p \geq 3$ is technically much more difficult with many subtleties that do not appear when $p=2$.
 \section{Nikishin systems on stars}\label{sectionnikishinsystems}
\subsection{Definition and basic properties of the Nikishin system}\label{subsectionnikishinsystems}
Let $p\geq 1$ be an integer, and let
\[
S_{\pm}:=\{z\in\mathbb{C}: z^{p+1}\in\mathbb{R}_{\pm}\},\qquad \mathbb{R}_{+}=[0,+\infty),\quad \mathbb{R}_{-}=(-\infty,0].
\]
Then
\[
S_-=e^{\frac{\pi i}{p+1}}S_+.
\]

We construct $p$ finite stars contained in $S_{\pm}$ as follows:
\begin{align*}
\Gamma_{j} & :=\{z\in\mathbb{C}: z^{p+1}\in[a_{j},b_{j}]\},\qquad \quad 0\leq
j\leq p-1,
\end{align*}
where
\[
0\leq a_{j}<b_{j}<\infty,\quad j\equiv 0\mod 2,
\]
\[
-\infty< a_{j}<b_{j}\leq 0, \quad j\equiv 1\mod 2,
\]
so that $\Gamma_{j}\subset S_{+}$ if $j$ is even, and $\Gamma_{j}\subset S_{-}$
if $j$ is odd. We assume throughout that $\Gamma_{j}\cap
\Gamma_{j+1}=\emptyset$ for all $0\leq j\leq p-2$, that is,
any two consecutive stars do not meet at the origin.

We define now a Nikishin system on $(\Gamma_{0},\ldots,\Gamma_{p-1})$. For each
$0\leq j\leq p-1$,  let $\sigma_{j}$ denote a positive, rotationally invariant (over the angle $2\pi/(p+1)$) measure on $\Gamma_{j}$, with infinitely many points in its support. These will
be the measures generating the Nikishin system.

Let
\[
\widehat{\mu}(x):=\int\frac{d\mu(t)}{x-t}
\]
denote the Cauchy transform of a complex measure $\mu$, and let
$\mu_1,\ldots,\mu_N$ be $N\geq 1$ measures such that $\mu_j$ and $\mu_{j+1}$
have disjoint supports for every $1\leq j\leq N-1$.
We define the measure $\langle \mu_{1},\ldots,\mu_{N}\rangle$ by the following
recursive procedure.
For $N=1$, $\langle \mu_{1}\rangle:= \mu_{1}$, for $N=2$,
\[
d\langle \mu_{1},\mu_{2}\rangle (x):= \widehat{\mu}_{2}(x)\,d\mu_{1}(x),
\]
and for $N>2$,
\[
\langle \mu_{1},\ldots,\mu_{N}\rangle := \langle \mu_{1},\langle
\mu_{2},\ldots,\mu_{N}\rangle\rangle.
\]

We then define the Nikishin system
$(s_{0},\ldots,s_{p-1})=\mathcal{N}(\sigma_0,\ldots,\sigma_{p-1})$ generated by
the vector of
$p$ measures $(\sigma_0,\ldots,\sigma_{p-1})$ by setting
\begin{equation}\label{def:sj}
s_{j}:=\langle \sigma_{0},\ldots, \sigma_{j}\rangle, \qquad 0\leq j\leq p-1.
\end{equation}
Notice that these measures $s_{j}$ are supported on the first star $\Gamma_{0}$.

It is convenient, however, to think of this Nikishin system as the first row of
the following hierarchy  of measures $s_{k,j}$,
\begin{align}\label{hiers}
 \begin{array}{ccccc}
  s_{0,0} & s_{0,1} & s_{0,2} & \cdots & s_{0,p-1}\\
  & s_{1,1} & s_{1,2} & \cdots & s_{1,p-1}\\
& & s_{2,2} & \cdots & s_{2,p-1}\\
 &  & & \ddots & \vdots\\
  &  & &  & s_{p-1,p-1}
 \end{array}
\end{align}
where
\begin{equation}\label{def:skj}
 s_{k,j}=\langle \sigma_k,\ldots,\sigma_j\rangle,\quad 0\leq k\leq j\leq p-1.
\end{equation}
More descriptively, the measures $s_{k,j}$ are inductively defined by setting
\begin{equation}\label{hiermeasskj}
\begin{split}
s_{k,k}:={}&\sigma_k, \quad 0\leq k\leq p-1,\\
d s_{k,j}(z)={}&\int_{\Gamma_{k+1}} \frac{ds_{k+1,j}(t)}{z-t}d\sigma_k(z),\quad
0\leq k < j\leq p-1.
\end{split}
\end{equation}

Notice then that for each pair $k$, $j$ with $0\leq k\leq j\leq p-1$,
$(s_{k,k},\ldots,s_{k,j})=\mathcal{N}(\sigma_k,\ldots,\sigma_j)$ is the
Nikishin system generated by
$(\sigma_k,\ldots,\sigma_j)$.

Throughout the paper we will use the notation
\[
\omega:=e^{\frac{2\pi i}{p+1}}.
\]
The following proposition summarizes several basic properties that will be
needed later.

\begin{proposition}\label{proposition1} For every $0\leq k\leq j\leq p-1$, the
measure $s_{k,j}$ satisfies the symmetry property
\begin{align}\label{rotinvarh}
ds_{k,j}(\omega z)=\omega^{k-j}ds_{k,j}(z).
\end{align}
Also, for every integrable $f$ on $\Gamma_k$, we have
\begin{equation}\label{symm:intF:1}
\int_{\Gamma_{k}}f(\omega z)\,d s_{k,j}(z)=\omega^{j-k}\,\int_{\Gamma_k}f(z)\,d
s_{k,j}(z),
\end{equation}
\begin{align}\label{eq14}
\int_{\Gamma_k}\overline{f(\overline{z})}ds_{k,j}(z)={} &
\overline{\int_{\Gamma_k}f(z)ds_{k,j}(z)}.
\end{align}
\end{proposition}
\begin{proof} The relation \eqref{rotinvarh} holds
trivially for $k=p-1$, and for $0\leq k<p-1$ it is proven by reverse induction on $k$. Formula \eqref{symm:intF:1} follows immediately from \eqref{rotinvarh}. The relation \eqref{eq14} is also proven 
by reverse induction on $k$, using \eqref{hiermeasskj} and that, due to its rotational invariance, 
$d\sigma_{k}(t)=d\sigma_k(\overline{t})$.\end{proof}

For every $0\leq j\leq p-1$, we shall denote by $\sigma^*_j$ the push-forward of $\sigma_j$ under the map $z\mapsto z^{p+1}$, that is,
 $\sigma^*_j$ is the measure on $[a_j,b_j]$ such that for every Borel set $E\subset [a_j,b_j]$,
\begin{equation}\label{def:sigma:star}
\sigma^*_j(E):=\sigma_j\left(\{z:z^{p+1}\in E\}\right).
\end{equation}

We now construct, out of these $\sigma_j^*$, a new hierarchy of measures
$\mu_{k,j}$, $0\leq k\leq j\leq p-1$:
\begin{align}\label{hiermu}
 \begin{array}{ccccc}
  \mu_{0,0} & \mu_{0,1} & \mu_{0,2} & \cdots & \mu_{0,p-1}\\
  & \mu_{1,1} & \mu_{1,2} & \cdots & \mu_{1,p-1}\\
& & \mu_{2,2} & \cdots & \mu_{2,p-1}\\
 &  & & \ddots & \vdots\\
  &  & &  & \mu_{p-1,p-1}
 \end{array}
\end{align}
where  the measures $\mu_{k,j}$ are inductively defined by setting
\begin{align}\label{def:mukj}
\begin{split}
\mu_{k,k}:={}&\sigma^*_k, \quad 0\leq k\leq p-1,\\
d\mu_{k,j}(\tau)={}&\left(\tau\int_{a_{k+1}}^{b_{k+1}}\frac{d\mu_{k+1,j}(s)}{\tau-s}
\right)d\sigma^*_{k}(\tau),\quad \tau \in [a_k,b_k],\quad 0\leq k < j\leq p-1.
\end{split}
\end{align}

In the following result we describe the relationship between the measures $\mu_{k,j}$ and $s_{k,j}$.

\begin{proposition} \label{proposition2} For every $0\leq k\leq j\leq p-1$, we
have
\begin{align*}
\int_{\Gamma_{k}}\frac{ds_{k,j}(t)}{z-t}={}&
z^{p+k-j}\int_{a_{k}}^{b_{k}}\frac{d\mu_{k,j}(\tau)}{z^{p+1}-\tau},
\end{align*}
that is,
\begin{align}\label{reductiontosegments}
\widehat{s}_{k,j}(z)=z^{p+k-j}\widehat{\mu}_{k,j}(z^{p+1}).
\end{align}
Hence, for every continuous function $f$ on $[a_k,b_k]$,
\begin{align}\label{changefromstartosegment}
\int_{a_k}^{b_k}f(\tau)d\mu_{k,j}(\tau)={}&
\int_{\Gamma_{k}}f(z^{p+1})z^{j-k}ds_{k,j}(z).
\end{align}
\end{proposition}
\begin{proof}
Using \eqref{rotinvarh} we find that $\int_{\Gamma_k}\frac{t^lds_{k,j}(t)}{z-t^{p+1}}=0$ when $l \not\equiv j-k\mod
(p+1)$. Hence, it follows that
\begin{align}\label{eq:skj}
\int_{\Gamma_k}\frac{ds_{k,j}(t)}{z-t}=z^{p+k-j}\int_{\Gamma_k}\frac{t^{j-k}ds_{k,j}(t)}{z^{p+1}-t^{p+1}}.
\end{align}
This already proves (\ref{reductiontosegments}) for $k=j$. If (\ref{reductiontosegments}) holds for some value of $k\in\{0,\ldots,j\}$, then using \eqref{eq:skj}, \eqref{def:mukj}, and \eqref{def:sigma:star}
yields that (\ref{reductiontosegments}) is also true for $k-1$. The proof of \eqref{changefromstartosegment} also follows by backward induction on $k$.\end{proof}

\subsection{Multiple orthogonal polynomials and functions of the second kind}\label{subsectionmultiorthogonal}

\begin{definition}\label{multiorthogonalitydef}
Let $\{Q_{n}(z)\}_{n=0}^\infty$ be the sequence of monic polynomials of lowest
degree that satisfy the following non-hermitian orthogonality conditions:
\begin{equation}\label{orthog:Qn}
\int_{\Gamma_{0}} Q_{n}(z)\,z^{l}\,d s_{j}(z)=0,\qquad l=0,\ldots,\left\lfloor
\frac{n-j-1}{p}\right\rfloor,\qquad 0\leq j\leq p-1.
\end{equation}
\end{definition}

In what follows we will use the notation
\[
d_{n}:=\deg(Q_{n}), \quad n\geq 0.
\]
Using \eqref{symm:intF:1}-\eqref{eq14}, one easily sees that the polynomials
$Q_{n}(z)$, $Q_{n}(\omega z)$ and $\overline{Q_{n}(\overline{z})}$ satisfy the
same orthogonality relations (\ref{orthog:Qn}).
Thus, by the uniqueness of $Q_{n}$, we have that
\begin{equation}\label{symm:Qn}
Q_{n}(\omega z)=\omega^{d_{n}}\,Q_{n}(z),\qquad
Q_{n}(z)=\overline{Q_{n}(\overline{z})}, \quad n\geq 0.
\end{equation}

Let $0\leq \ell\leq p$ be such that $d_n\equiv \ell \mod (p+1)$, so that
\[
d_n= d(p+1)+\ell, \quad d:=\left\lfloor \frac{d_n}{p+1}\right\rfloor.
\]
Then, the first relation in \eqref{symm:Qn} implies that
\begin{equation}\label{Qd}
Q_n(t)=t^\ell \mathcal{Q}_d(t^{p+1})
\end{equation}
for some polynomial $ \mathcal{Q}_d$ of exact degree $d$.

We note that in this paper we use the following standard notations for a real number $x$:
\begin{align*}
\lfloor x \rfloor & =\sup\{m\in\mathbb{Z}: m\leq x\},\\
\lceil x \rceil & =\inf\{m\in\mathbb{Z}: m\geq x\}.
\end{align*}

The polynomials $Q_n$ are intrinsically related to the so-called functions of
the second kind, which we define next.

\begin{definition}\label{definitionPsi}
Set $\Psi_{n,0}=Q_{n}$ and let
\[
\Psi_{n,k}(z)=\int_{\Gamma_{k-1}}\frac{\Psi_{n,k-1}(t)}{z-t}\,d\sigma_{k-1}(t),
\qquad k=1,\ldots,p.
\]
\end{definition}

Observe that $\Psi_{n,k}$ is analytic in $\overline{\mathbb{C}}\setminus
\Gamma_{k-1}$.
Our next proposition shows that the function $\Psi_{n,k}$
satisfies multiple orthogonality conditions similar to those satisfied by $Q_n$
but with respect to the  Nikishin system given by the $k$th row of the
hierarchy \eqref{hiers}. Note that the function $\Psi_{n,p}$ is excluded from this proposition.

\begin{proposition}\label{orthoPsi}
For each $k=0,\ldots,p-1$, the function $\Psi_{n,k}$ satisfies the following
orthogonality conditions
\begin{equation}\label{orthog:Psink}
\int_{\Gamma_{k}} \Psi_{n,k}(z)\,z^{l}\,ds_{k,j} (z)=0,\qquad 0\leq l\leq \lf
\frac{n-j-1}{p} \rf,\qquad k\leq j\leq p-1.
\end{equation}
\end{proposition}
\begin{proof}
The proof follows by induction on $k$, based on the equality 
\begin{align*}
 \int_{\Gamma_{k+1}} \Psi_{n,k+1}(z)\,z^{l}\,ds_{k+1,j}(z)=\int_{\Gamma_{k}} \Psi_{n,k}(t)\,p_{l}(t)\,d\sigma_{k}(t)-\int_{\Gamma_{k}}
\Psi_{n,k}(t)\,t^{l}\,ds_{k,j} (t),
\end{align*}
where $p_{l}$ denotes the polynomial $p_{l}(t)=\int_{\Gamma_{k+1}}\frac{z^{l}-t^{l}}{z-t}\,ds_{k+1,j} (z)$.
\end{proof}

\begin{proposition}
The functions $\Psi_{n,k}$ satisfy the symmetry property
\begin{equation}\label{symm:Psink}
\Psi_{n,k}(\omega z)=\omega^{d_n-k}\,\Psi_{n,k}(z),\qquad k=0,\ldots,p, \quad
n\geq 0,
\end{equation}
where, as above, $d_n$ is the degree of $Q_n$.
\end{proposition}
\begin{proof}
The proof is again by induction on $k$. The case $k=0$ is the already proved
symmetry property \eqref{symm:Qn} for the polynomials $Q_{n}$. Assuming that
\eqref{symm:Psink} holds for $k$, then
\[
\Psi_{n,k+1}(\omega z)=\int_{\Gamma_{k}}\frac{\Psi_{n,k}(t)}{\omega
z-t}\,d\sigma_{k}(t)=\int_{\Gamma_{k}}\frac{\Psi_{n,k}(\omega t)}{\omega
z-\omega t}\,d\sigma_{k}(t)=\omega^{d_n-k-1}\,\Psi_{n,k+1}(z).
\]\end{proof}

We now seek to find  an analogue of the polynomial $\mathcal{Q}_d$ in
\eqref{Qd} for the functions $\Psi_{n,k}$. To accomplish that, we first need the
following representation.
\begin{proposition}
Assume that $d_n \equiv \ell \mod (p+1)$ with $0\leq \ell\leq p$. Then, for
each $k=1,\ldots, p$ we have
\begin{equation}\label{altintrep:Psink}
\Psi_{n,k}(z)=
z^{p-s}\,\int_{\Gamma_{k-1}}\frac{\Psi_{n,k-1}(t)\,t^{s}}{z^{p+1}-t^{p+1}}\,
d\sigma_{k-1}(t),
\end{equation}
where $s$ is the only integer in $\{0,\ldots,p\}$ such that $s\equiv
k-1-\ell\mod (p+1)$, that is,
\begin{equation}\label{param:s}
s=\begin{cases}
k-1-\ell, & \ell< k,\\
p+k-\ell, & k \leq \ell.
\end{cases}
\end{equation}
\end{proposition}
\begin{proof}
Let $1\leq k\leq p$. From Definition \ref{definitionPsi}  we deduce that
\[
\Psi_{n,k}(z)=\sum_{l=0}^{p}z^{p-l}\int_{\Gamma_{k-1}}\frac{\Psi_{n,k-1}(t)\,t^{l}}{z^{p+1
}-t^{p+1}}\,d\sigma_{k-1}(t).
\]
By \eqref{symm:Psink}, the integral of the $l$th-term of this sum equals $0$ if $d_n-k+1+l \not\equiv 0 \mod (p+1)$, which proves  
\eqref{altintrep:Psink}.
\end{proof}

\begin{definition}\label{definitionpsi}
Set $\psi_{n,0}:=\mathcal{Q}_d$, and for $1\leq k\leq p$, let
$\psi_{n,k}$ be the function analytic in $\mathbb{C}\setminus
[a_{k-1},b_{k-1}]$ defined as
\[
\psi_{n,k}(z)=\begin{cases}
z\int_{\Gamma_{k-1}}\frac{\Psi_{n,k-1}(t)\,t^{k-1-\ell}}{z-t^{p+1}}\,d\sigma_{
k-1}(t), & \ell< k,\\[1em]
\int_{\Gamma_{k-1}}\frac{\Psi_{n,k-1}(t)\,t^{p+k-\ell}}{z-t^{p+1}}\,d\sigma_{k-1
}(t), & k\leq \ell.
\end{cases}
\]
\end{definition}

This definition is what one naturally gets by
substituting the expressions in \eqref{param:s} for $s$ in
\eqref{altintrep:Psink}, and doing so also yields at once the following
corollary.
\begin{corollary}\label{coroPsi} Suppose $d_n
\equiv \ell \mod (p+1)$ with $0\leq \ell\leq p$, and define
\begin{equation}\label{varymeas:sigma}
d\sigma_{n,k}(\tau):=\begin{cases}
d\sigma_{k}^{*}(\tau), & \ell\leq k,\\
\tau\,d\sigma_{k}^{*}(\tau), & k<\ell.
\end{cases}
\end{equation}Then,
\begin{equation}\label{modifiedPsink}
z^{k-\ell}\,\Psi_{n,k}(z)=\psi_{n,k}(z^{p+1}), \quad 0\leq k\leq p,
\end{equation}
and for all $1\leq k\leq p$,
\begin{equation}
\psi_{n,k}(z)=\begin{cases}\label{recurrenceforpsipequena}
z\int_{a_{k-1}}^{b_{k-1}}\frac{\psi_{n,k-1}(\tau)}{z-\tau}\,d\sigma_{n,k-1}
(\tau), & \ell< k,\\[1em]
\int_{a_{k-1}}^{b_{k-1}}\frac{\psi_{n,k-1}(\tau)}{z-\tau}\,d\sigma_{n,k-1}(\tau), & k\leq \ell.
\end{cases}
\end{equation}
\end{corollary}

We have seen that the functions $\Psi_{n,k}$ satisfy orthogonality relations
with respect to the hierarchy \eqref{hiers}. We now show that their associated
functions $\psi_{n,k}$ do the same with respect to the hierarchy
\eqref{hiermu}.

\begin{proposition}
Let $0\leq k\leq p-1$ and assume that $d_n\equiv \ell \mod (p+1)$ with $0\leq
\ell\leq p$. Then the function $\psi_{n,k}$ satisfies the following
orthogonality conditions:
 \begin{equation}\label{orthogredPsink}
\int_{a_{k}}^{b_{k}}\psi_{n,k}(\tau)\,\tau^{s}\,d\mu_{k,j}(\tau)=0,\quad
\left\lceil\frac{\ell-j}{p+1}\right\rceil\leq s\leq \left\lfloor
\frac{n+p\ell-1-j(p+1)}{p(p+1)}\right\rfloor,\quad k\leq j\leq p-1.
\end{equation}
\end{proposition}

\begin{proof}
Let $0\leq k\leq j\leq p-1$. From the orthogonality conditions
\eqref{orthog:Psink} and the relations \eqref{modifiedPsink} and
\eqref{changefromstartosegment} we get that
\begin{align}\label{preorthocondpsi}
\begin{split}
\int_{a_k}^{b_k}\psi_{n,k}(\tau)\tau^{\frac{\ell+l-j}{p+1}}\, d\mu_{k,j}(\tau)=0,\quad 0\leq l\leq \lf \frac{n-j-1}{p}\rf,\quad \ell+l-j \equiv 0 \mod
(p+1).
\end{split}
\end{align}
If we take $l$ in \eqref{preorthocondpsi} satisfying $\ell+l-j \equiv 0 \mod
(p+1)$, then we can write $l=j-\ell+s(p+1)$ and we obtain the orthogonality
conditions
\begin{equation}\label{mod:orthogPsink}
\int_{a_k}^{b_k}\psi_{n,k}(\tau)\tau^s\, d\mu_{k,j}(\tau)=0,\quad
\left\lceil\frac{\ell-j}{p+1}\right\rceil\leq s\leq
\left\lfloor\frac{1}{p+1}\left\lfloor
\frac{n+p\ell-1-j(p+1)}{p}\right\rfloor\right\rfloor.
\end{equation}

Since $
\left\lfloor\frac{\left\lfloor x\right\rfloor}{p+1}\right\rfloor=
\left\lfloor\frac{x}{p+1}\right\rfloor$ for all $x\in \mathbb{R}$,
the range for $s$ in
(\ref{mod:orthogPsink}) takes the form in \eqref{orthogredPsink}.
\end{proof}

\subsection{Counting the number of orthogonality conditions}\label{counting}

\begin{definition}
Let  $n$ and $\ell$ be nonnegative integers with $0\leq\ell\leq p$. For each
$0\leq j\leq p-1$, let $M_j=M_j(n,\ell)$ be the number of integers $s$
satisfying the inequalities
\begin{equation}\label{countingcond}
 \left\lceil\frac{\ell-j}{p+1}\right\rceil\leq s\leq \left\lfloor
\frac{n+p\ell-1-j(p+1)}{p(p+1)}\right\rfloor.
\end{equation}
For each $0\leq k\leq p-1$, we define
\begin{align}\label{def:Znk}
 Z(n,k)=Z(n,\ell,k):=\sum_{j=k}^{p-1}M_j.
\end{align}
Also, we convene to set $Z(n,p):=0$. 
\end{definition}

Herafter we shall always write $Z(n,k)$ instead of $Z(n,\ell,k)$ because in 
all future situations the number $\ell$ will be dependent on $n$. We easily 
deduce from \eqref{countingcond}-\eqref{def:Znk} that 
\begin{equation}\label{asympformZnk}
Z(n,k)=\frac{n(p-k)}{p(p+1)}+O(1),\qquad n\rightarrow\infty.
\end{equation}

It is also clear from the definition that for every $n$ and $\ell$,
\[
Z(n,k)\geq Z(n,k+1), \quad 0\leq k\leq p-2,
\]
and
\begin{equation}\label{differenceznk}
Z(n,k)-Z(n,k+1)=\#\left\{s:\left\lceil\frac{\ell-k}{p+1}\right\rceil\leq s\leq
\left\lfloor \frac{n+p\ell-1-k(p+1)}{p(p+1)}\right\rfloor\right\}.
\end{equation}
Hence, choosing $j=k$ in \eqref{orthogredPsink}, and noticing that
\begin{equation}\label{lowerbound}
 \left\lceil\frac{\ell-j}{p+1}\right\rceil=\begin{cases}
  0, & \mathrm{if}\ \ell\leq j,\\
1, &\mathrm{if} \ j<\ell,
 \end{cases}
\end{equation}
we arrive at the following corollary.
\begin{corollary}
Let $0\leq k\leq p-1$ and assume that $d_n\equiv \ell \mod (p+1)$ with $0\leq
\ell\leq p$. Then the function $\psi_{n,k}$ satisfies the  orthogonality
conditions
 \begin{equation}\label{orthogredPsink2}
\int_{a_{k}}^{b_{k}}\psi_{n,k}(\tau)\,\tau^{s}\,d\sigma_{n,k}(\tau)=0,\quad
0\leq s\leq  Z(n,k)-Z(n,k+1)-1.
\end{equation}
\end{corollary}

Let us fix nonnegative integers $n$ and $\ell$, with $\ell$ satisfying $
0\leq \ell\leq p$. We associate to $n$ and $\ell$ three numbers $\alpha$, $\beta$,
and $v$, letting $\alpha$ and $\beta$ be, respectively,
the quotient and the remainder in the division of $n+p\ell-1$ by $p(p+1)$, and
letting $v$ be the quotient in the division of $\beta$ by $p+1$. That is,
\[
n+p\ell-1=\alpha p(p+1)+\beta,\quad 0\leq \beta\leq p(p+1)-1,
\]
\begin{equation}\label{defalpha}
\alpha=\lf\frac{n+p\ell-1}{p(p+1)}\rf,\quad
v=\lf\frac{\beta}{p+1}\rf.
\end{equation}
Notice that
\[
 0\leq v\leq p-1.
\]

\begin{lemma}\label{lemcounting1}
If $ \ell\leq v$,  the inequality (\ref{countingcond}) is equivalent to
\begin{align}\label{conds1}
\begin{array}{ll}
1\leq s\leq \alpha,\ &\mathrm{if}\  0\leq j<\ell,\\
0\leq s\leq \alpha,\ &\mathrm{if}\  \ell\leq j\leq v,\\
0\leq s\leq \alpha-1,\ &\mathrm{if}\   v<j\leq p-1,
\end{array}
\end{align}
while if $v< \ell$, then (\ref{countingcond}) is equivalent to
\begin{align}\label{conds2}
\begin{array}{ll}
& 1\leq s\leq \alpha, \quad \mathrm{if}\  0\leq j\leq v,\\
&1\leq s\leq \alpha-1,\quad \mathrm{if}\   v<j<\ell,\\
& 0\leq s\leq \alpha-1,\quad \mathrm{if}\  \ell\leq j\leq p-1.
\end{array}
\end{align}
Moreover,
\begin{equation}\label{eq:Znk}
Z(n,k)=\begin{cases}
\lc\frac{n-\ell}{p+1}\rc-k\alpha, & \mathrm{if}\ \ k\leq \ell,v,\\[1em]
\lc\frac{n-\ell}{p+1}\rc-k\alpha+\ell-v-1, & \mathrm{if}\ \  \ell,v<k,\\[1em]
\lc\frac{n-\ell}{p+1}\rc-k(\alpha+1)+\ell, & \mathrm{if}\ \ 0\leq \ell<k\leq
v,\\[1em]
\lc\frac{n-\ell}{p+1}\rc-k(\alpha-1)-v-1, & \mathrm{if}\ \ 0\leq v<k\leq
\ell.\\[1em]
\end{cases}
\end{equation}

\end{lemma}

\begin{proof}
We begin by writing (\ref{countingcond}) in the form
\begin{align}\label{eq3}
\left\lceil\frac{\ell-j}{p+1}\right\rceil\leq s\leq \left\lfloor \alpha
+\frac{\beta-j(p+1)}{p(p+1)}\right\rfloor.
\end{align}
Now, since
\begin{equation}\label{eq5}
0\leq v\leq p-1,\quad v(p+1)\leq \beta<(v+1)(p+1),
\end{equation}
we have
\[
-1<-\frac{(p-1)}{p}\leq \frac{\beta-j(p+1)}{p(p+1)}\leq \frac{p(p+1)-1}{p(p+1)}<1,
\]
and since $\alpha$ is an integer, this implies that
\begin{equation}\label{upperbound}
 \left\lfloor \alpha +\frac{\beta-j(p+1)}{p(p+1)}\right\rfloor=\begin{cases}
  \alpha, & \mathrm{if}\ 0\leq j\leq v,\\
  \alpha-1, &\mathrm{if} \ v<j\leq p-1.
 \end{cases}
  \end{equation}
The inequalities \eqref{conds1}-\eqref{conds2} follow from (\ref{lowerbound})
and (\ref{upperbound}).

As for \eqref{eq:Znk}, we shall only prove it for the case $k\leq\ell,v $,
since the remaining cases listed in \eqref{eq:Znk} are proven similarly.  For every $j\in \{k,\ldots,p-1\}$, we can use \eqref{upperbound} to express, in terms of $\alpha$, the number $M_j$ of integer values $s$ that satisfy \eqref{eq3}. With this expressions at hand, one easily finds that if $k\leq  \ell, v$, then
\begin{align*}
\begin{split}
Z(n,k)=\sum_{j=k}^{p-1}M_j=
\alpha p+v-\ell+1-k\alpha.
\end{split}
\end{align*}
Now, using the expression that defines $\alpha$ in (\ref{defalpha}), we find
\begin{align*}
 \alpha p+v-\ell+1=\frac{n-\ell +(v+1)(p+1)-(\beta+1)}{p+1},
\end{align*}
which together with (\ref{eq5}) yields
\[
\frac{n-\ell}{p+1}\leq  \alpha p+v-\ell+1\leq
\frac{n-\ell}{p+1}+\frac{p}{p+1}.
\]
Being $ \alpha p+v-\ell+1$ an integer, we deduce that it must be equal to $ \left\lceil\frac{n-\ell}{p+1}\right\rceil$.\end{proof}

\subsection{AT-system property}
The system of continuous functions $u_1(x),\ldots,u_n(x)$ is said to be an algebraic Chebyshev system (AT-system) over the interval $[a,b]$ for the set of integers $(d_1,\ldots,d_n)$ ($d_j\geq 0$), if for any choice of polynomials $(P_1(x),\ldots,P_n(x))\not=(0,0,\ldots,0)$, with $\deg(P_j)\leq d_j-1$, the polynomial combination
\[
P_1(x)u_1(x)+\cdots+P_n(x)u_n(x)
\]
has at most $d_1+\cdots+d_n-1$ zeros on $[a,b]$. Here and in what follows, a polynomial of degree $-1$ is understood to be the constant zero function.

Since $\mu_{k,k}=\sigma_k^*$ and
$d\mu_{k,j}(t)=t\widehat{\mu}_{k+1,j}(t)d\sigma_k^*(t)$ for $k<j\leq p-1$, the
orthogonality conditions \eqref{orthogredPsink} can be equivalently written as
$\psi_{n,k}$ being orthogonal to polynomial linear combinations
of functions of the form
\[
1,t\widehat{\mu}_{k+1,k+1}(t),\ldots, t\widehat{\mu}_{k+1,m}(t),
\widehat{\mu}_{k+1,m+1}(t),\ldots,\widehat{\mu}_{k+1,p-1}(t),
\]
for some $k\leq m\leq  p-1$. We now  prove that
any such collection of functions forms an AT-system over $[a_k,b_k]$.

\begin{proposition}\label{lemma:ATsystem} Let $k$, $m$ be integers such that
$0\leq k\leq m\leq  p-1$. For each $j$ in the range $k\leq j\leq p-1$, let
$P_{j}$ be a polynomial of degree at most $d_j-1$, with $d_j\geq 0$, and
suppose that
\[
d_k\geq d_{k+1}\geq \cdots \geq d_{m}\geq d_{m+1}-1\geq d_{m+2}-1\geq
\cdots\geq d_{p-1}-1.
\]
If $(P_k,\ldots,P_{p-1})\not=(0,0,\ldots,0)$, then
\begin{equation}\label{eq:combCauchy}
H(z)=P_k(z)+\sum_{k+1\leq j\leq m}P_j(z)z\widehat{\mu}_{k+1,j}(z)+ \sum_{m< j
\leq p-1}P_j(z)\widehat{\mu}_{k+1,j}(z)
\end{equation}
 has at most $D_H:=\sum_{j=k}^{p-1}d_j-1$ zeros in $[a_{k},b_k]$.
\end{proposition}

\begin{proof} The proof is by induction on $k$. If $k=p-1$, the statement is
trivially true, as in this case we simply have $H(z)=P_{p-1}(z)$ and
$D_H=d_{p-1}-1$. Assume that the thesis of the proposition is also true for
$k+1$, $0<k+1\leq p-1$, but that for the value $k$, there is a corresponding
function $H$ of the form (\ref{eq:combCauchy}) with at least $D_H+1$ zeros in
$[a_k,b_k]$.

Then, for this $H$ not all the polynomials $P_j$ corresponding to $k+1\leq
j\leq p-1$, can be simultaneously zero. Let $T$ be a monic polynomial that
vanishes at the zeros of $H$ in $[a_k,b_k]$, and let $\gamma$ be
a positively oriented simple contour around the interval $[a_{k+1},b_{k+1}]$
that leaves outside the zeros of $T$. Since $H/T$ is analytic outside $[a_{k+1},b_{k+1}]$ and 
\begin{align}\label{eq:HT1}
 \begin{split}
\frac{H(z)}{T(z)}={}& O\left(\frac{1}{z^{D_H+2-d_k}}\right),\quad z\to\infty,
\end{split}
\end{align}
we have 
\begin{align}\label{int:HoverT}
\frac{1}{2\pi i}\int_{\gamma} \frac{z^{u}\,H(z)}{T(z)}\,dz=0,\quad u=0,\ldots,D_H-d_k.
\end{align}

By definition, $\widehat{\mu}_{k+1,j}(z)=\int_{a_{k+1}}^{b_{k+1}}\frac{d\mu_{k+1,j}(\tau)}{z-\tau
}$, and  so an application of Fubini's theorem and Cauchy's integral formula give
\begin{align}\label{FubiniCauchy}
\frac{1}{2\pi i}\int_{\gamma}
\frac{z^{u}P_j(z)}{T(z)}\widehat{\mu}_{k+1,j}(z)dz=\int_{a_{k+1}}^{b_{k+1}}\frac{\tau^{u}P_j(\tau)}{T(\tau)}d\mu_{k+1,j}(\tau),\quad u\in\mathbb{N}\cup\{0\}.
\end{align}
If we now replace $H$ in \eqref{int:HoverT} by the right-hand side of (\ref{eq:combCauchy}), we get from \eqref{FubiniCauchy} and \eqref{def:mukj}
 that
\begin{align}\label{orthocondG1}
\int_{a_{k+1}}^{b_{k+1}}\tau^{u}G(\tau)\frac{\tau^{1-\delta_{mk}}d\sigma^*_{k+1}(\tau)}{T(\tau)}= 0,\quad
u=0,\ldots,D_H-d_k, &
\end{align}
where $\delta_{mk}$ is the Kronecker delta and $G$ is the function that, in case $m=k$, is given by
\[
G(z)=P_{k+1}(z)+\sum_{k+2\leq j\leq p-1}P_j(z)z\widehat{\mu}_{k+2,j}(z),
\]
while if $k+1\leq m\leq p-1$, 
\[
G(z)=P_{k+1}(z)+\sum_{k+2\leq j\leq m} z P_j(z)\widehat{\mu}_{k+2,j}(z)+\sum_{m<
j\leq p-1}P_j(z)\widehat{\mu}_{k+2,j}(z).
\]
By induction hypothesis,  the function $G$ has at most
$D_G=D_H-d_k$ zeros in $[a_{k+1},b_{k+1}]$, which together with
\eqref{orthocondG1} implies that $G$ must be identically
zero, yielding a contradiction, since at least one of the $P_j$'s for $k+1\leq
j\leq p-1$ is not identically zero.
\end{proof}

\begin{corollary}\label{lemma:orthosystem} Let $k$, $m$ be integers such that
$0\leq k\leq m\leq  p-1$. Let $\{d_j\}_{j=k}^{p-1}$ be a finite sequence of
nonnegative integers such that
\[
d_k\geq d_{k+1}\geq \cdots \geq d_{m}\geq d_{m+1}-1\geq d_{m+2}-1\geq
\cdots\geq d_{p-1}-1.
\]
Suppose $F\not\equiv 0$ is a function analytic and real-valued on $[a_k,b_k]$,
satisfying the orthogonality conditions
\begin{align}\label{orthocondgeneral1}
 \int_{a_k}^{b_k}F(\tau)\tau^{s+\delta}d\mu_{k,j}(\tau)=0,\quad 0\leq s\leq d_j-1,\quad
k\leq j\leq m,
\end{align}
\begin{align}\label{orthocondgeneral2}
 \int_{a_k}^{b_k}F(\tau)\tau^{s}d\mu_{k,j}(\tau)=0,\quad 0\leq s\leq d_j-1,\quad m<
j\leq p-1,
\end{align}
where the constant $\delta=1$ if $m<p-1$ and $d_{m+1}=d_m+1$, otherwise
$\delta$ could be taken to be either $1$ or $0$. Then, $F$ has at least
\[
N:=\sum_{j=k}^{p-1}d_j
\]
zeros of odd multiplicity in $(a_{k},b_k)$.
\end{corollary}

\begin{proof} The orthogonality conditions \eqref{orthocondgeneral1} and
\eqref{orthocondgeneral2} imply that
\begin{align}\label{orthocondgeneralH}
 \int_{a_k}^{b_k}\tau^\delta F(\tau)H(\tau)d\sigma^*_k(z)=0
\end{align}
for every $H$ of the form
\[
H(z)=P_k(z)+\sum_{k+1\leq j\leq m}P_j(z)z\widehat{\mu}_{k+1,j}(z) + \sum_{m< j
\leq p-1}P_j(z)\widehat{\mu}_{k+1,j}(z),
\]
where $P_j$ is a polynomial of degree at most $d_j-1$ for each $k\leq j\leq
p-1$, and if $\delta=0$ then we take $m=p-1$. Applying Proposition \ref{lemma:ATsystem}, we see that any
such function $H$, not identically zero, has at most $N-1$ zeros in $[a_k,b_k]$. Consequently, if $F$ had a number $D<N$ of zeros with odd multiplicity in
$(a_k,b_k)$, say $x_1,\ldots,x_D$, we could find $H$ with simple zeros at these
$x_k$'s, and with $N-D-1$ zeros (counting multiplicities) at the endpoints of
the interval $[a_k,b_k]$. Since $H$ does not admit any more zeros on that
closed interval, the integral \eqref{orthocondgeneralH} cannot be zero.
Therefore, $D\geq N$.
\end{proof}

\subsection{Normality of the Nikishin system and zeros of $Q_n$}

The Nikishin system of measures $(s_0,\ldots,s_{p-1})$ is said to be normal provided that the degree of the multi-orthogonal polynomial $Q_n$ is maximal,
that is, $d_n=n$ for all $n\geq 0$.
We will prove this normality in this section.

\begin{proposition}\label{prop:oddmult}
Let $n$, $k$, and $\ell$ be nonnegative integers satisfying $0\leq k\leq
p-1$, and $d_n\equiv\ell\mod(p+1)$, $0\leq \ell\leq p$. Then, the function
$\psi_{n,k}$ has at least $Z(n,k)$ zeros with odd multiplicity in the open
interval $(a_{k},b_{k})$. In particular, if follows that
\[
d_n=n,\quad n\geq 0,
\]
that is, the polynomial $Q_n$ has degree $n$, and the associated
polynomial
$\mathcal{Q}_d$ has exactly
\[
d=\frac{n-\ell}{p+1}
\]
zeros,  which are all simple and located in $(a_0,b_0)$.
\end{proposition}
\begin{proof}
The  total number of orthogonality conditions in \eqref{orthogredPsink} is given by  
$Z(n,k)$ as defined in  \eqref{def:Znk}.  Using Lemma \ref{lemcounting1}, these orthogonality conditions can be more specifically written, but since $j$ varies in \eqref{orthogredPsink} only from $k$ to $p-1$, we split the analysis in the following five cases: 1) $k\leq \ell\leq v\leq
p-1$, 2) $k\leq v< \ell\leq p$, 3) $0\leq \ell,v<k$, 4) $0\leq \ell<k\leq v$, and 5) $v<k\leq \ell$. In each of these cases, we use \eqref{conds1} and \eqref{conds2} to write, for each $j\in \{k,\ldots,p-1\}$, the relation \eqref{orthogredPsink} as an orthogonality of $\psi_{n,k}$ against $\tau^s$ or against $\tau^{s+1}$, according to whether the range of $s$ in \eqref{conds1}-\eqref{conds2} starts from $s=0$ or from $s= 1$, respectively. For instance, in the first case $k\leq \ell\leq v\leq
p-1$, \eqref{orthogredPsink} takes the form
\begin{align*}
\int_{a_{k}}^{b_{k}}\psi_{n,k}(\tau)\,\tau^{s+1}\,d\mu_{k,j}(\tau) & =0, \qquad
0\leq s\leq \alpha-1,\quad k\leq j<\ell,\notag\\
\int_{a_{k}}^{b_{k}}\psi_{n,k}(\tau)\,\tau^{s}\,d\mu_{k,j}(\tau) & =0, \qquad
0\leq s\leq \alpha,\quad \ell\leq j\leq v,\\
\int_{a_{k}}^{b_{k}}\psi_{n,k}(\tau)\,\tau^{s}\,d\mu_{k,j}(\tau) & =0, \qquad
0\leq s\leq \alpha-1,\quad v<j\leq p-1.\notag
\end{align*}

Written as such, in each of the five cases the orthogonality relations \eqref{orthogredPsink} adhere to the form of \eqref{orthocondgeneral1}-\eqref{orthocondgeneral2}, and so  Corollary \ref{lemma:orthosystem} ensures that $\psi_{n,k}$ has at least $Z(n,k)$  zeros of odd multiplicity in $(a_k,b_k)$. Particularly, for $k=0$,  $\psi_{n,0}=\mathcal{Q}_d$ has at least
$Z(n,0)=\lc\frac{n-\ell}{p+1}\rc$ zeros of odd multiplicity in $(a_0,b_0)$.
Hence,
\[
 \lc\frac{n-\ell}{p+1}\rc\leq d=\frac{d_n-\ell}{p+1}\leq \frac{n-\ell}{p+1},
\]
finishing the proof of the proposition.
\end{proof}

\subsection{On the difference $Z(n,j)-Z(n,j+1)$}

We now prove an auxiliary result that will be useful later.

\begin{lemma}
 Let $n$ be any nonnegative integer. Suppose  $n\equiv\ell\mod (p+1)$ and $n\equiv r\mod
p$, $0\leq \ell\leq p$, $0\leq r\leq p-1$, and let
\begin{align}\label{def:lambda}
\lambda=\lf\frac{n}{p(p+1)}\rf.
\end{align}
Then, for every $j=0,\ldots,p-1$ we have
\begin{equation}\label{differenceznj}
Z(n,j)-Z(n,j+1)= \begin{cases}
                      \lambda,&  j<\ell\leq r,\\
                      \lambda+1,& \ell\leq j<  r,\\
                      \lambda, & \ell\leq r\leq j,\\
                      \lambda+1,& j< r<\ell,\\
                      \lambda,& r\leq j<\ell,\\
                       \lambda+1,&  r<\ell\leq j.
                     \end{cases}
\end{equation}
\end{lemma}
\begin{proof}
Let us write $n=mp +r$ with $m\geq 0$, and
\[
 m=\lambda (p+1)+q,\quad 0\leq q\leq p,
\]
so that 
\begin{equation}\label{ellkandq}
n=\lambda p(p+1)+pq+r,\quad  0\leq pq+r\leq p(p+1)-1,
\end{equation}
and  $\lambda$  is given by \eqref{def:lambda}. In terms of these quantities,  the remainder $\ell$ in the division of $n$ by $p+1$ is given by
\begin{align}\label{nlambdaq}
\ell=\begin{cases}
r-q,\ & \ell\leq r,\\
p+1+r-q,\ &  r<\ell.
\end{cases}
\end{align}

Combining \eqref{differenceznk}, \eqref{ellkandq}, and \eqref{nlambdaq}, we get that for all $0\leq j\leq p-1$,
\begin{align*}
Z(n,j)-Z(n,j+1)={}&\lambda+\left\lfloor\frac{
(r-j)(p+1)-1}{p(p+1)}\right\rfloor+ \left(\begin{cases}
                      0,& \ell\leq r\\
                      1,& r< \ell
                     \end{cases}\right)+\left(
                     \begin{cases}
                      1,& \ell\leq j\\
                      0,& j< \ell
                     \end{cases}\right),
\end{align*}
thus \eqref{differenceznj} follows.\end{proof}

\subsection{Order of decay and zeros of the functions of the second kind}
\begin{proposition}
Let $1\leq k\leq p$, and suppose that $n\equiv\ell\mod(p+1)$. Then, as
$z\to\infty$,
\begin{align}\label{decayinfpsink}
\psi_{n,k}(z) & =O(z^{-N(n,k)}),
\end{align}
where
\[
 N(n,k)=\begin{cases}
         Z(n,k-1)-Z(n,k), & \ell<k,\\
         Z(n,k-1)-Z(n,k)+1, & k\leq \ell,
        \end{cases}
\]
recall $Z(n,p)=0$.
\end{proposition}
\begin{proof}
From \eqref{recurrenceforpsipequena} and \eqref{varymeas:sigma} we see that if $\ell<k$, the
Laurent expansion of $\psi_{n,k}$ at infinity has the form
\begin{align}\label{expansionforpsi2}
 \psi_{n,k}(z)=&\sum_{s=0}^\infty
z^{-s}\int_{a_{k-1}}^{b_{k-1}}\psi_{n,k-1}(\tau)\tau^{s}d\sigma_{k-1}^*(\tau),
 \end{align}
while if $k\leq \ell$, the expansion of $\psi_{n,k}$ at infinity
is as in \eqref{expansionforpsi2} but with the series starting from $s= 1$ (instead of from $s=0$). Combining these expansions with\eqref{orthogredPsink2}   yields
\eqref{decayinfpsink}.
\end{proof}

We are now in position to prove the following result.

\begin{proposition}\label{prop:exactnumberzeros}
For each $n\geq 0$ and $k=0,\ldots,p-1$, the function $\psi_{n,k}$ has exactly $Z(n,k)$
zeros in $\mathbb{C}\setminus([a_{k-1},b_{k-1}]\cup\{0\})$; they are all simple
and lie in the open interval $(a_{k},b_{k})$. The function $\psi_{n,p}$ has no
zeros in $\mathbb{C}\setminus([a_{p-1},b_{p-1}]\cup\{0\})$.
\end{proposition}
\begin{proof}
The proof is by induction on $k$. It was already shown in Proposition
\ref{prop:oddmult} that the polynomial $\psi_{n,0}=\mathcal{Q}_d$ has degree
$Z(n,0)=\frac{n-\ell}{p+1}$, all its zeros are simple and lie in the interval
$(a_{0},b_{0})$.

Let us assume that the result holds for $k-1$, $k\geq 1$, but $\psi_{n,k}$
has at least $Z(n,k)+1$ zeros in
$\mathbb{C}\setminus([a_{k-1},b_{k-1}]\cup\{0\})$, counting multiplicities. Let
$P_{n,k}(z)$ denote the monic polynomial whose zeros are the zeros of
$\psi_{n,k}$ in $\mathbb{C}\setminus([a_{k-1},b_{k-1}]\cup\{0\})$. Since
$\psi_{n,k}(\overline{z})=\overline{\psi_{n,k}(z)}$, the complex zeros of
$\psi_{n,k}$, if any, must come in conjugate pairs, so $P_{n,k}$ is a
polynomial with real coefficients with $\deg (P_{n,k})\geq Z(n,k)+1$.

Let us momentarily set $\mathcal{P}_{n,k}:=zP_{n,k}$ if $\ell<k$, and $\mathcal{P}_{n,k}:=P_{n,k}$ otherwise. Then, it follows from \eqref{decayinfpsink} that
\begin{equation}\label{estimateinf:1}
\frac{\psi_{n,k}(z)}{\mathcal{P}_{n,k}(z)}=O\Big(\frac{1}{z^{Z(n,k-1)+2}}\Big),\qquad
z\rightarrow\infty,
\end{equation}
and this function is analytic outside $[a_{k-1},b_{k-1}]$. Let $\gamma$ be a
closed Jordan curve that surrounds $[a_{k-1},b_{k-1}]$ and leaves the zeros of
$P_{n,k}$ outside.
Then, it follows from \eqref{estimateinf:1} and \eqref{recurrenceforpsipequena} that for $j=0,\ldots,Z(n,k-1)$, we
have
\begin{align}\label{orthcondPnk}
0 & =\frac{1}{2\pi i}\int_{\gamma}\frac{z^{j}\,\psi_{n,k}(z)}{\mathcal{P}       _{n,k}(z)}\,dz
=\int_{a_{k-1}}^{b_{k-1}}\psi_{n,k-1}(\tau)\,\tau^{j}\,\frac{d\sigma_{n,k-1}
(\tau)}{P_{n,k}(\tau)},
\end{align}
where we applied Cauchy's theorem, Cauchy's integral formula and Fubini's
theorem. The above orthogonality conditions of $\psi_{n,k-1}$ with respect to
the measure $\frac{d\sigma_{n,k-1}(\tau)}{P_{n,k}(\tau)}$ imply that
 $\psi_{n,k-1}$ has at least $Z(n,k-1)+1$ zeros in $(a_{k-1},b_{k-1})$,
contrary to our initial hypothesis.

Thus, the function $\psi_{n,k}$ has at most $Z(n,k)$ zeros in
$\mathbb{C}\setminus([a_{k-1},b_{k-1}]\cup\{0\})$. This together with
Proposition \ref{prop:oddmult} gives the result.
\end{proof}

For the asymptotic analysis that will be performed later it is crucial to
consider the polynomials whose zeros coincide with those of the functions
$\psi_{n,k}$. We introduce now a notation for these polynomials.

\begin{definition}
For any integers $n\geq 0$ and $k$ with $0\leq k\leq p-1$, let $P_{n,k}$ denote
the monic polynomial whose zeros are the zeros of $\psi_{n,k}$ in
$(a_{k},b_{k})$. For convenience we also define the polynomials $P_{n,-1}\equiv
1$, $P_{n,p}\equiv 1$.
\end{definition}

Hence by Proposition \ref{prop:exactnumberzeros} we know that $P_{n,k}$ has
degree $Z(n,k)$ and all its zeros are simple. Note that $P_{n,0}=\psi_{n,0}$.

\begin{proposition}
Let $0\leq k\leq p-1$. Then,
the function $\psi_{n,k}$ satisfies the following orthogonality conditions:
\begin{equation}\label{varyorthog:psink:1}
\int_{a_{k}}^{b_{k}}\psi_{n,k}(\tau)\,\tau^{s}\,\frac{d\sigma_{n,k}(\tau)}{P_{n,
k+1}(\tau)}=0,\qquad s=0,\ldots,Z(n,k)-1.
\end{equation}
\end{proposition}
\begin{proof}
For $0\leq k\leq p-2$, these orthogonality conditions are just \eqref{orthcondPnk}. For
$k=p-1$, \eqref{varyorthog:psink:1} follows from \eqref{orthogredPsink}
and \eqref{eq:Znk}, since $\lc\frac{\ell-(p-1)}{p+1}\rc=0$ if $\ell\leq p-1$
and $\lc\frac{\ell-(p-1)}{p+1}\rc=1$ if $\ell=p$.
\end{proof}

\begin{corollary}\label{cor:zeros}
Let $0\leq k\leq p-1$, and let $I$ be any connected component of
$[a_{k},b_{k}]\setminus \supp(\sigma_{k}^{*})$. Then the polynomial $P_{n,k}$
has at most one zero in the closure of $I$.
\end{corollary}
\begin{proof}
Suppose that $P_{n,k}$ has two distinct zeros $\tau_{1}$ and $\tau_{2}$ in
$\overline{I}$ and assume that $\ell\leq k$ (the case $k<\ell$ follows along the same lines). Set $L_{n,k}(\tau):=\frac{P_{n,k}(\tau)}{(\tau-\tau_{1})(\tau-\tau_{2})}$. Then,  according to
\eqref{varyorthog:psink:1} we have
\begin{equation}\label{eq:zerocont}
\int_{a_{k}}^{b_{k}}\psi_{n,k}(\tau)\,L_{n,k}(\tau)\frac{d\sigma_{k}^{*}(\tau)}{P_{n,k+1}(\tau)}=0,
\end{equation}
since $L_{n,k}$ is a polynomial of degree $Z(n,k)-2$. On the other hand, the function $\psi_{n,k}L_{n,k}$ has
constant sign and finitely many zeros on $\supp(\sigma_{k}^{*})$, therefore its
integral with respect to the measure $d\sigma_{k}^{*}(\tau)/P_{n,k+1}(\tau)$
should be different from zero, contradicting \eqref{eq:zerocont}.
\end{proof}

\subsection{The auxiliary functions $H_{n,k}$}

We now introduce certain functions that will play an important role
in the analysis that will follow.

\begin{definition}\label{defHnk}For integers $n\geq 0$ and $0\leq k\leq p$,
set
\begin{equation}\label{eq:def:Hnk}
H_{n,k}(z):=\frac{P_{n,k-1}(z)\,\psi_{n,k}(z)}{P_{n,k}(z)}.
\end{equation}
\end{definition}

Note that $H_{n,0}\equiv 1$. Since the zeros of $P_{n,k}$ are zeros of
$\psi_{n,k}$ outside $[a_{k-1},b_{k-1}]$, we have
\[
H_{n,k}\in\mathcal{H}(\mathbb{C}\setminus[a_{k-1},b_{k-1}]),\quad 1\leq k\leq
p.
\]

Putting together \eqref{varymeas:sigma}, \eqref{varyorthog:psink:1}, and \eqref{eq:def:Hnk},
 we readily obtain
the following result.

\begin{proposition}\label{orthoPnk}
For any $k=0,\ldots,p-1$, the polynomial $P_{n,k}$ satisfies the following
orthogonality conditions:
\begin{equation}\label{varyorthog:Pnk}
\int_{a_{k}}^{b_{k}}P_{n,k}(\tau)\,\tau^{s}\,\frac{H_{n,k}(\tau)\,d\sigma_{n,k}
(\tau)}{P_{n,k-1}(\tau)\,P_{n,k+1}(\tau)}=0,\qquad s=0,\ldots,Z(n,k)-1.
\end{equation}
Recall that $P_{n,-1}, P_{n,p}\equiv 1$.
\end{proposition}

We prove now a formula analogous to \eqref{recurrenceforpsipequena} for the
functions $H_{n,k}$.

\begin{proposition}
Let $1\leq k\leq p$ and $n\equiv \ell \mod(p+1)$, $0\leq \ell\leq p$. Then,
\begin{equation}\label{intrep:Hnk}
H_{n,k}(z)=\begin{cases}
z\int_{a_{k-1}}^{b_{k-1}}\frac{P_{n,k-1}^{2}(\tau)}{z-\tau}\,\frac{H_{n,k-1}
(\tau)\,d\sigma_{n,k-1}(\tau)}{P_{n,k-2}(\tau)\,P_{n,k}(\tau)}, & \ell<k,\\[1em]
\int_{a_{k-1}}^{b_{k-1}}\frac{P_{n,k-1}^{2}(\tau)}{z-\tau}\,\frac{H_{n,k-1}
(\tau)\,d\sigma_{n,k-1}(\tau)}{P_{n,k-2}(\tau)\,P_{n,k}(\tau)}, &
k\leq \ell.\\
\end{cases}
\end{equation}
\end{proposition}
\begin{proof}
We know by \eqref{varyorthog:Pnk} that for any polynomial $Q$ with $\deg(Q)\leq
Z(n,k-1)$, $1\leq k\leq p$, we have
\begin{equation}\label{eq:orthog:Q}
\int_{a_{k-1}}^{b_{k-1}}\frac{Q(z)-Q(\tau)}{z-\tau}\,P_{n,k-1}(\tau)\,\frac{H_{n
,k-1}(\tau)\,d\sigma_{n,k-1}(\tau)}{P_{n,k-2}(\tau)\,P_{n,k}(\tau)}=0.
\end{equation}
If we take in particular $Q=P_{n,k-1}$ in \eqref{eq:orthog:Q}, then we obtain
\[
P_{n,k-1}(z)\int_{a_{k-1}}^{b_{k-1}}\frac{P_{n,k-1}(\tau)H_{n,
k-1}(\tau)\,d\sigma_{n,k-1}(\tau)}{(z-\tau)P_{n,k-2}(\tau)\,P_{n,k}(\tau)}=\int_{a_{k-1}
}^{b_{k-1}}\frac{P_{n,k-1}^{2}(\tau)H_{n,k-1}(\tau)\,d\sigma_{n
,k-1}(\tau)}{(z-\tau)P_{n,k-2}(\tau)\,P_{n,k}(\tau)}.
\]

Since $Z(n,k)\leq Z(n,k-1)$, we can also apply \eqref{eq:orthog:Q} for $Q=P_{n,k}$, which together with this last equality yields that for $k=1,\ldots,p,$
\begin{equation}\label{eqaux1}
\frac{1}{P_{n,k}(z)}\int_{a_{k-1}}^{b_{k-1}}\frac{\psi_{n,k-1}(\tau)}{z-\tau}\,
d\sigma_{n,k-1}(\tau)=\frac{1}{P_{n,k-1}(z)}\,\int_{a_{k-1}}^{b_{k-1}}\frac{P_{n
,k-1}^{2}(\tau)}{z-\tau}\,\frac{H_{n,k-1}(\tau)\,d\sigma_{n,k-1}(\tau)}{P_{n,k-2
}(\tau)\,P_{n,k}(\tau)},
\end{equation}
and  the result follows from \eqref{eqaux1}, \eqref{recurrenceforpsipequena} and \eqref{eq:def:Hnk}.
\end{proof}

In what follows, we shall use the notation $\sign(f,I)$ to mean the sign of the
function $f$ on the interval $I$, and $\Delta_k$ shall denote the open interval
$(a_k,b_k)$.
\begin{corollary}
Let $1\leq k\leq p-1$ and $n\equiv \ell \mod(p+1)$, $0\leq \ell\leq p$. Then,
with the convention that $Z(n,-1)=0$, we have
\begin{equation}\label{signHnk}
\sign(H_{n,k},\Delta_k)=\begin{cases}
(-1)^{(k+1)[Z(n,k-2)-Z(n,k)]}\sign(H_{n,k-1},\Delta_{k-1}),\ & \ell<k,\\[0.5em]
(-1)^{1+(k+1)[Z(n,k-2)-Z(n,k)]}\sign(H_{n,k-1},\Delta_{k-1}),\ & k\leq \ell.
\end{cases}
\end{equation}
\end{corollary}
\begin{proof}
Suppose first that $k\leq \ell$. Then, by \eqref{intrep:Hnk} and
\eqref{varymeas:sigma},
\[
H_{n,k}(z)=\int_{\Delta_{k-1}}\frac{P_{n,k-1}^{2}(\tau)}{z-\tau}\,\frac{H_{n,k-1
}(\tau)\tau d\sigma^*_{k-1}(\tau)}{P_{n,k-2}(\tau)\,P_{n,k}(\tau)}.
\]
If $k$ is even, $\Delta_{k-1}$ lies in $(-\infty,0]$, while $\Delta_{k-2}$ and
$\Delta_{k}$ lie in $[0,\infty)$. Since the monic polynomials $P_{n,k-2}$,
$P_{n,k}$ have their zeros in $\Delta_{k-2}$, $\Delta_{k}$, respectively, and
$\deg(P_{n,k})=Z(n,k)$, the above equality gives
\begin{align*}
\sign(H_{n,k},\Delta_k)=(-1)^{Z(n,k-2)-Z(n,k)+1}\,\sign(H_{n,k-1},\Delta_{k-1}).
\end{align*}
If $k$ is odd, $\Delta_{k-1}$ lies in $[0,\infty)$, $P_{n,k-2}$ and $P_{n,k}$
are both positive in $\Delta_{k-1}$, so that
\begin{align*}
\sign(H_{n,k},\Delta_k)=-\sign(H_{n,k-1},\Delta_{k-1}).
\end{align*}
Suppose now that $\ell<k$. Then, by \eqref{intrep:Hnk} and
\eqref{varymeas:sigma},
\begin{align*}
H_{n,k}(z)=z\int_{\Delta_{k-1}}\frac{P_{n,k-1}^{2}(\tau)}{z-\tau}\,\frac{H_{n,
k-1}(\tau)\,d\sigma^*_{k-1}(\tau)}{P_{n,k-2}(\tau)\,P_{n,k}(\tau)},
\end{align*}
so that if $k$ is even, $\sign(H_{n,k},\Delta_k)=(-1)^{Z(n,k-2)-Z(n,k)}\,\sign(H_{n,k-1},\Delta_{k-1})$, 
while for $k$ odd, $\sign(H_{n,k},\Delta_k)=\sign(H_{n,k-1},\Delta_{k-1})$.
\end{proof}
\section{Recurrence relation and positivity of the recurrence coefficients}\label{sectionrecurrence}

\begin{proposition}\label{prop:recurrencerelationforQn}
The polynomials $Q_{n}$ satisfy the following three-term recurrence relation of
order $p+1$:
\begin{equation}\label{threetermrec}
z Q_{n}(z)=Q_{n+1}(z)+a_{n}\,Q_{n-p}(z),\qquad n\geq p,\qquad
a_{n}\in\mathbb{R},
\end{equation}
where
\begin{equation}\label{eq:firstQns}
Q_{\ell}(z)=z^{\ell},\qquad \ell=0,\ldots,p.
\end{equation}
\end{proposition}
\begin{proof}
The equation \eqref{eq:firstQns} is clear since we know that if
$n=d(p+1)+\ell$, $0\leq \ell\leq p$, then $Q_n(z)=z^\ell\mathcal{Q}_d(z^{p+1})$
for some monic polynomial $\mathcal{Q}_d$ of degree $d$. Moreover, this also
implies that for $n\geq p$, $zQ_{n}(z)-Q_{n+1}(z)=c_{n-p}z^{n-p}+\cdots$. Thus,
we can write
\begin{equation}\label{expansion:rec}
z Q_{n}=Q_{n+1}+\sum_{j=0}^{n-p} b_j\,Q_j,
\end{equation}
for some real coefficients $\{b_j\}_{j=0}^{n-p}$. The goal is to show that
\begin{equation}\label{zeroreccoeff:1}
b_{0}=b_{1}=\cdots=b_{n-p-1}=0.
\end{equation}

Assume that $n=mp+k$, $0\leq k\leq p-1$. If we integrate \eqref{expansion:rec}
term by term with respect to the first measure $s_{0}$ of the Nikishin system,
we observe that the only non-vanishing integral is $\int Q_{0}\, ds_{0}$, and
consequently $b_{0}=0$. Integrating \eqref{expansion:rec} successively with
respect to $s_{j}$ we obtain $b_{j}=0$ for $j=1,\ldots,p-1$.

In general, one proves inductively that for all $l$ such that $0\leq l\leq
m-2$, we have
\begin{equation}\label{zeroreccoeff:3}
b_{lp}=b_{lp+1}=\cdots=b_{lp+(p-1)}=0.
\end{equation}
The case $l=0$ was described above. Assume now that all coefficients $b_{s}$ in
\eqref{expansion:rec} are zero for $s<lp$. If we multiply \eqref{expansion:rec}
by $z^{l}$ and integrate with respect to $s_{0}$, then the only non-vanishing
integral in the resulting expression is $\int z^{l}\,Q_{lp}(z)\,d s_{0}(z)$.
Indeed, all other integrals vanish because of the orthogonality conditions, and
 $\int z^{l}\,Q_{lp}(z)\,d s_{0}(z)=0$ would imply that $Q_{lp+1}$ and $Q_{lp}$
satisfy the same orthogonality conditions, violating the normality of the
Nikishin system. So $b_{lp}=0$. Integrating successively with respect to the
rest of the measures $s_{j}$ one obtains \eqref{zeroreccoeff:3}. The remaining
part of \eqref{zeroreccoeff:1} is
\[
b_{(m-1)p}=b_{(m-1)p+1}=\cdots=b_{(m-1)p+k-1}=0,
\]
which is proved multiplying by $z^{m-1}$ and integrating with respect to
$s_{0},\ldots,s_{k-1}$.
\end{proof}
We now show that the functions of the second kind satisfy a similar recurrence
relation.
\begin{proposition}\label{prop:recurrencerelationpsin} Let $a_n$, $n\geq p$, be the coefficients of the recurrence
relation \eqref{threetermrec}. For every $n\geq p$, $0\leq k\leq p$, we have
\begin{equation}\label{threetermrecsecondkind}
z\Psi_{n,k}(z)=\Psi_{n+1,k}(z)+a_n\Psi_{n-p,k}(z),
\end{equation}
and if $n\equiv\ell\mod(p+1)$, $0\leq \ell \leq p-1$, then
\begin{equation}\label{threetermrecsecondkindpequena1}
\psi_{n,k}(z)=\psi_{n+1,k}(z)+a_n\psi_{n-p,k}(z),
\end{equation}
while if $n\equiv p\mod(p+1)$, then
\begin{equation}\label{threetermrecsecondkindpequena2}
z\psi_{n,k}(z)=\psi_{n+1,k}(z)+a_n\psi_{n-p,k}(z).
\end{equation}
\end{proposition}
\begin{proof}
For $k=0$, by definition, $\Psi_{n,0}=Q_n$, and so
\eqref{threetermrecsecondkind} reduces to \eqref{threetermrec}. Let us then
assume that \eqref{threetermrecsecondkind} holds for some $0\leq k\leq p-1$.
Then, by the very definition of $\Psi_{n,k}$, we have
\begin{align*}
\int_{\Gamma_k}\frac{t\Psi_{n,k}(t)d\sigma_k(t)}{z-t}={}&\int_{\Gamma_k}\frac{
\Psi_{n+1,k}(t)d\sigma_k(t)}{z-t}+a_n\int_{\Gamma_k}\frac{
\Psi_{n-p,k}(t)d\sigma_k(t)}{z-t}\\
={}& \Psi_{n+1,k+1}(z)+a_n\Psi_{n-p,k+1}(z).
\end{align*}
Now, from Proposition \ref{orthoPsi}, we know that
\[
\int_{\Gamma_k}\Psi_{n,k}(z)z^ld\sigma_k(z)=0, \quad 0\leq l\leq
\lf\frac{n-k-1}{p}\rf,
\]
and since $n-k-1\geq p-1-k\geq 0$,  we have that
$\int_{\Gamma_k}\Psi_{n,k}(z)d\sigma_k(z)=0$ and so
\[
\int_{\Gamma_k}\frac{t\Psi_{n,k}(t)d\sigma_k(t)}{z-t}=z\int_{\Gamma_k}\frac{
\Psi_{n,k}(t)d\sigma_k(t)}{z-t}-\int_{\Gamma_k}\Psi_{n,k}(t)d\sigma_k(t)=z\Psi_{
n,k+1}(z).
\]
Now, using \eqref{modifiedPsink} in \eqref{threetermrecsecondkind} we find that
for $n\geq p$, $0\leq k\leq p$,
\[
z^{\ell_n+1-k}\psi_{n,k}(z^{p+1})=z^{\ell_{n+1}-k}\psi_{n+1,k}(z^{p+1})+a_nz^{
\ell_{n-p}-k}\psi_{n-p,k}(z^{p+1}).
\]
where we are using the notation $\ell_n$ to mean the remainder of $n$ in the
division by $p+1$, $0\leq \ell_n\leq p$.

The relations
\eqref{threetermrecsecondkindpequena1}-\eqref{threetermrecsecondkindpequena2}
then follow from the fact that $\ell_{n+1}=\ell_{n-p}=\ell_n+1$ when
$\ell_n\leq p-1 $, while $\ell_{n+1}=\ell_{n-p}=0$ when $\ell_n=p$.\end{proof}

\begin{lemma} \label{psinplusonesigmank} Let $n\geq p$ and suppose that $n\equiv k \mod p$, $0\leq k\leq
p-1$, and $n\equiv\ell\mod(p+1)$, $0\leq \ell \leq p$. Then
\begin{align}\label{orthoznkdiff1}
\int_{a_k}^{b_k}\psi_{n+1,k}(\tau)\tau^{Z(n,k)-Z(n,k+1)}d\sigma_{n,k}(\tau)=0,
\quad \ell\leq p-1,
\end{align}
\begin{align}\label{orthoznkdiff2}
\int_{a_k}^{b_k}\psi_{n+1,k}(\tau)\tau^{Z(n,k)-Z(n,k+1)-1}d\sigma_{n,k}(\tau)=0,
\quad \ell=p,
\end{align}
and
\begin{align}\label{orthoznkdiff5}
\tau^{Z(n-p,k)-Z(n-p,k+1)}d\sigma_{n-p,k}(\tau)=\begin{cases}
\tau^{Z(n,k)-Z(n,k+1)}d\sigma_{n,k}(\tau),&  \ell\leq p-1,\\
\tau^{Z(n,k)-Z(n,k+1)-1}d\sigma_{n,k}(\tau), & \ell=p.
                                                \end{cases}
\end{align}
 \end{lemma}
 \begin{proof}
 Obviously,
\begin{align}\label{remaindern+1byp}
 n+1\equiv\begin{cases}
 \ell+1\mod (p+1), & \ell\leq p-1,\\
 0\mod (p+1), & \ell=p.\\
\end{cases}
\end{align}
With this in mind, we readily get from \eqref{varymeas:sigma} that
\eqref{orthoznkdiff1}-\eqref{orthoznkdiff2} are equivalent to
\begin{align}\label{orthoznkdiff3}
\int_{a_k}^{b_k}\psi_{n+1,k}(\tau)\tau^{Z(n,k)-Z(n,k+1)}d\sigma_{n+1,k}(\tau)=0,
\quad \ell\not=k,
\end{align}
 \begin{align}\label{orthoznkdiff4}
\int_{a_k}^{b_k}\psi_{n+1,k}(\tau)\tau^{Z(n,k)-Z(n,k+1)-1}d\sigma_{n+1,k}
(\tau)=0,\quad \ell=k.
\end{align}

Suppose $n\geq p$, $n\equiv k\mod p$, $0\leq k\leq p-1$. Then,
\begin{align}\label{remaindern+1byp+1}
 n+1\equiv\begin{cases}
 k+1\mod p, & k<p-1,\\
 0\mod p, & k=p-1.\\
\end{cases}
\end{align}

We then use \eqref{differenceznj} to analyze all possible cases emanating from
\eqref{remaindern+1byp} and \eqref{remaindern+1byp+1}. Using the notation
\[
\lambda(n)=\lf\frac{n}{p(p+1)}\rf,
\]
we find that
\begin{align*}
Z(n,k)-Z(n,k+1)=\lambda(n),
\end{align*}
\begin{align*}
Z(n+1,k)-Z(n+1,k+1) =\begin{cases}\lambda(n+1), & \ell=k,\\
                      \lambda(n+1), & \ell=p, \  k=p-1,\\
\lambda(n+1)+1,  & \mathrm{otherwise}.
                     \end{cases}
\end{align*}
Now, $\lambda(n+1)=\lambda(n)$ if $n+1$ is not a multiple of $p(p+1)$, and $\lambda(n+1)=\lambda(n)+1$ otherwise. The latter case holds exactly when $\ell=p$ and $k=p-1$. We then conclude that in all instances the exponent of $\tau$ in
\eqref{orthoznkdiff3}-\eqref{orthoznkdiff4} equals $Z(n+1,k)-Z(n+1,k+1)-1$. Notice that this quantity is non-negative since the smallest integer $n\geq p$ satisfying that $\ell=k$ (i.e.,
having the same remainder when divided by $p$ and by $p+1$) is $n=p(p+1)$. This together with \eqref{orthogredPsink2} yields
\eqref{orthoznkdiff1}-\eqref{orthoznkdiff2}.

Now, both $n$ and $n-p$ leave the same remainder $k$ when they are divided by $p$. If
$\ell\leq p-1$, then $n-p\equiv\ell+1\mod (p+1)$, and \eqref{differenceznj} yields
 \[
 Z(n-p,k)-Z(n-p,k+1)=\begin{cases}
 Z(n,k)-Z(n,k+1), & \ell\not=k,\\
 Z(n,k)-Z(n,k+1)-1, & \ell=k,\\
\end{cases}
\]
while from \eqref{varymeas:sigma} we get
\[
d\sigma_{n-p,k}(\tau) =\begin{cases}
d\sigma_{n,k}(\tau), & \ell\not=k,\\
\tau d\sigma_{n,k}(\tau), & \ell=k.\\
\end{cases}
\]
So we see that \eqref{orthoznkdiff5} holds in the case $\ell\leq p-1$. Similarly, if $\ell=p$, then $n-p\equiv 0\mod(p+1)$,
so that, again by \eqref{differenceznk} and \eqref{varymeas:sigma}, we have
\[
 Z(n-p,k)-Z(n-p,k+1)=Z(n,k)-Z(n,k+1), \quad \ell=p,
\]
and
\[
\tau d\sigma_{n-p,k}(\tau) =d\sigma_{n,k}(\tau), \quad \ell=p.
\]
\end{proof}
\begin{lemma} \label{lemmainproof}
Suppose that $n\equiv \ell\mod (p+1)$, $0\leq \ell\leq p$, and that $n=mp+k$
with $0\leq k\leq p-1$. With the notation $\Delta_k:=(a_k,b_k)$, we have
\begin{align}\label{signPnkminus1Pnkplus1formula}
\sign(P_{n,k-1}P_{n,k+1},\Delta_k)=\begin{cases}
1,\ &\mathrm{if}\ k\ \mathrm{is\ even},\\
-1,\ &\mathrm{if}\ k\not=\ell,\ k\ \mathrm{odd},\\
1,\ &\mathrm{if}\ k=\ell,\ k\ \mathrm{odd}.
\end{cases}
\end{align}
Also, for every $j$ in the range $0\leq j\leq k$, we have
\begin{align}\label{signHnkformula}
\sign(H_{n,j},\Delta_{j})=\begin{cases}(-1)^{j},\ &  j\leq \ell,\\
1,\ &  \ell<j.
\end{cases}
\end{align}
\end{lemma}
\begin{proof}From \eqref{signHnk}, we know that for every $1\leq j\leq p-1$,
\begin{equation}\label{intrep:Hnklemma2kt}
\sign(H_{n,j},\Delta_{j})=\begin{cases}
(-1)^{1+(j+1)[Z(n,j-2)-Z(n,j)]}\sign(H_{n,j-1},\Delta_{j-1}),\ & j\leq \ell,\\
(-1)^{(j+1)[Z(n,j-2)-Z(n,j)]}\sign(H_{n,j-1},\Delta_{j-1}),\ & \ell<j,
\end{cases}
\end{equation}
which will allow us to recursively compute the sign of $H_{n,j}$.

From \eqref{differenceznj}, we get that for all $j< k$,
\begin{align*}
Z(n,j)-Z(n,j+1)={}& \begin{cases}
                      \lambda+1,& \ell\leq j\ \mathrm{or}\ k<\ell,\\
                      \lambda,& j<\ell\leq  k,
                     \end{cases}
\end{align*}
while if $k\leq  j$, then
\begin{align*}
Z(n,j)-Z(n,j+1)={}& \begin{cases}
                      \lambda,& \ell\leq k\ \mathrm{or}\ j<\ell,\\
                      \lambda+1,& k<\ell\leq  j.
                     \end{cases}
\end{align*}
Since
\[
Z(n,j-2)-Z(n,j)= Z(n,j-2)-Z(n,j-1)+Z(n,j-1)-Z(n,j),
\]
this implies that if $2\leq j\leq k$, then
\begin{align}\label{differenceznktminus2}
Z(n,j-2)-Z(n,j)={}& \begin{cases}
                      2(\lambda+1), &\ell<j-1\ \mathrm{or}\ k< \ell,\\
                        2\lambda+1,&  \ell=j-1,\\
                       2\lambda, &  j-1< \ell\leq k,
                     \end{cases}
\end{align}
and if $j= k+1$, then
\begin{align}\label{differenceznktminus2:2}
Z(n,k-1)-Z(n,k+1)={}& \begin{cases}
                      2\lambda+1, &\ell\not= k,\\
                       2\lambda, &  \ell= k.
                     \end{cases}
\end{align}

The validity of \eqref{signPnkminus1Pnkplus1formula} for $k$ even
 is trivial, since in such a case, $\Delta_k\subset (0,\infty)$ while the zeros
of the monic polynomials $P_{n,k\pm 1}$ all lie in $\Delta_{k\pm 1}\subset
(-\infty, 0)$.

If $k\geq 1$ is odd, then $\Delta_k\subset (-\infty,0)$,
$\Delta_{k-1},\Delta_{k+1}\subset (0,\infty)$, so that
\begin{align*}
\sign(P_{n,k-1}P_{n,k+1},\Delta_k)={}& (-1)^{Z(n,k-1)-Z(n,k+1)}
               \end{align*}
and \eqref{signPnkminus1Pnkplus1formula} for $k$ odd follows from
\eqref{differenceznktminus2:2}.

Now, directly from \eqref{intrep:Hnklemma2kt} we get
\begin{align}\label{signhn1hn0}
 \sign(H_{n,1},\Delta_{1})=\begin{cases}
                             \sign(H_{n,0},\Delta_{0}),&\ell=0,\\
                             -\sign(H_{n,0},\Delta_{0}),&\ell\geq 1,
                            \end{cases}
\end{align}
while from \eqref{differenceznktminus2} and \eqref{intrep:Hnklemma2kt}, we
obtain that for all $2\leq j\leq k$,
\begin{equation}\label{signhnkt}
\sign(H_{n,j},\Delta_{j})=\begin{cases}
-\sign(H_{n,j-1},\Delta_{j-1}), &  k<\ell,\\
\sign(H_{n,j-1},\Delta_{j-1}), & \ell< j-1,\\
(-1)^{j+1}\sign(H_{n,j-1},\Delta_{j-1}), & \ell= j-1,\\
-\sign(H_{n,j-1},\Delta_{j-1}), & j\leq \ell\leq k.
\end{cases}
\end{equation}

This implies that
\begin{align*}
\sign(H_{n,j},\Delta_{j})=-\sign(H_{n,j-1},\Delta_{j-1}),\quad 1\leq j\leq \ell,
\end{align*}
 and iterating this relation we obtain (recall that $H_{n,0}\equiv 1$)
\begin{align}\label{signktleql}
\sign(H_{n,j},\Delta_{j})=(-1)^{j},\quad 0\leq j\leq \ell.
\end{align}

We now get from \eqref{signhnkt} and  \eqref{signktleql}  that if $\ell<j\leq
k$ and $j\geq 2$, then
\begin{align*}
\sign(H_{n,j},\Delta_{j})=\sign(H_{n,\ell+1},\Delta_{\ell+1})=(-1)^{\ell}
\sign(H_{n,\ell},\Delta_{\ell})=1.
\end{align*}
By \eqref{signhn1hn0}, we see that this last relation also holds if
$j=1>\ell=0$, completing the proof of \eqref{signHnkformula}.
\end{proof}

\begin{theorem}\label{thm:positivityrecurrencecoeff} The coefficients $a_n$ of the recurrence relation
\eqref{threetermrec} are all positive, i.e., $a_n> 0$ for every $n\geq p$.
\end{theorem}
\begin{proof}
It follows directly from
\eqref{threetermrecsecondkindpequena1}-\eqref{threetermrecsecondkindpequena2}
and Lemma \ref{psinplusonesigmank} that for all $n\geq p$, with $n\equiv k\mod p$, we
have
\begin{align}\label{twointegralswitha_n}
\int_{a_k}^{b_k}\tau^{Z(n,k)-Z(n,k+1)}\psi_{n,k}(\tau)d\sigma_{n,k}
(\tau)=a_n\int_{a_k}^{b_k}\tau^{Z(n-p,k)-Z(n-p,k+1)}\psi_{n-p,k}(\tau)d\sigma_{
n-p,k}(\tau).
\end{align}
Since $\mathrm{deg}(P_{n,k})=Z(n,k)$, we get from \eqref{varyorthog:psink:1}
that
\begin{align}\label{repreformomentpsink}
\int_{a_k}^{b_k}\tau^{Z(n,k)-Z(n,k+1)}\psi_{n,k}(\tau)d\sigma_{n,k}(\tau)={} &
\int_{a_k}^{b_k}\tau^{Z(n,k)-Z(n,k+1)}P_{n,k+1}(\tau)\frac{\psi_{n,k}(\tau)}{P_{
n,k+1}(\tau)}d\sigma_{n,k}(\tau)\notag\\
={}
&\int_{a_k}^{b_k}P^2_{n,k}(\tau)\frac{H_{n,k}(\tau)}{P_{n,k-1}(\tau)P_{n,k+1}
(\tau)}d\sigma_{n,k}(\tau).
\end{align}
It follows from  Lemma \ref{lemmainproof} and \eqref{varymeas:sigma} that if
$n=mp+k$, then
\begin{align*}
\sign\left(\int_{a_k}^{b_k}P^2_{n,k}(\tau)\frac{H_{n,k}(\tau)}{P_{n,k-1}(\tau)P_
{n,k+1}(\tau)}d\sigma_{n,k}(\tau)\right)=(-1)^k,
\end{align*}
and since $n-p=(m-1)p+k$, we conclude that the two integrals in
\eqref{twointegralswitha_n} have the same sign, and thus $a_n>0$.
\end{proof}

\begin{corollary}\label{cor:zerointerlacing}
The non-zero roots of the polynomials $Q_{n}$ and $Q_{n+1}$ interlace on $\Gamma_{0}$ for every $n\geq p+1$, i.e., between two consecutive non-zero roots of $Q_{n}$ there is exactly one non-zero root of $Q_{n+1}$ and vice versa.
\end{corollary}
\begin{proof}
This interlacing property is a consequence of \eqref{threetermrec}--\eqref{eq:firstQns} and
the positivity of the recurrence coefficients, as it was shown in Theorem 2.2 from \cite{Romd}.
\end{proof}

We remark that for each $0\leq k\leq p-1$, the zeros of the polynomials $P_{n,k}$ and $P_{n+1,k}$ interlace on $(a_{k},b_{k})$. This result is proved in \cite[Theorem 3.1]{LopLop}. 

\section{Normalization}

In this section we introduce a convenient normalization of the polynomials
$P_{n,k}$ and the functions $H_{n,k}$.

It follows from the definition of the functions $H_{n,k}$ and the polynomials
$P_{n,k}$ that the measures
\[
\frac{H_{n,k}(\tau)\,d\sigma_{n,k}(\tau)}{P_{n,k-1}(\tau)\,P_{n,k+1}(\tau)},\qquad 0\leq k\leq p-1,
\]
have constant sign on the interval $[a_{k}, b_{k}]$. We then denote by
\[
\frac{|H_{n,k}(\tau)|\,d|\sigma_{n,k}|(\tau)}{|P_{n,k-1}(\tau)\,P_{n,k+1}(\tau)|
}
\]
the positive normalization of this measure and we have
\begin{equation}\label{eq:normvaryorthog:Pnk}
\int_{a_{k}}^{b_{k}}P_{n,k}(\tau)\,\tau^{s}\,\frac{|H_{n,k}(\tau)|\,d|\sigma_{n,
k}|(\tau)}{|P_{n,k-1}(\tau)\,P_{n,k+1}(\tau)|}=0,\quad
s=0,\ldots,Z(n,k)-1,\quad k=0,\ldots,p-1.
\end{equation}

Let
\begin{align}
K_{n,-1} & :=1,\quad K_{n,p}:=1,\label{eq:def:Knm1}\\
K_{n,k} & :=
\left(\int_{a_{k}}^{b_{k}}P_{n,k}^{2}(\tau)\,\frac{|H_{n,k}(\tau)|\,d|\sigma_{n,
k}|(\tau)}{|P_{n,k-1}(\tau)\,P_{n,k+1}(\tau)|}\right)^{-1/2},\qquad
k=0,\ldots,p-1,\label{eq:def:Knk}
\end{align}
and we also define the constants
\begin{equation}\label{def:kappank}
\kappa_{n,k}:=\frac{K_{n,k}}{K_{n,k-1}},\qquad k=0,\ldots,p.
\end{equation}

\begin{definition}
For $k=0,\ldots,p,$ we define
\begin{align}
p_{n,k} & :=\kappa_{n,k}\,P_{n,k},\label{eq:def:lpnk}\\
h_{n,k} & := K_{n,k-1}^{2}\,H_{n,k},\label{eq:def:hnk}
\end{align}
where the constants $\kappa_{n,k}$ and $K_{n,k}$ are given in
\eqref{def:kappank} and \eqref{eq:def:Knm1}--\eqref{eq:def:Knk}, respectively.
\end{definition}

We will denote by $\nu_{n,k}$ the measure on $[a_{k},b_{k}]$ given by
\[
d\nu_{n,k}(\tau):=\frac{h_{n,k}(\tau)\,d\sigma_{n,k}(\tau)}{P_{n,k-1}(\tau)\,P_{
n,k+1}(\tau)},\qquad k=0,\ldots,p-1.
\]
Again this measure has constant sign in $[a_{k},b_{k}]$, and we will denote by
$\varepsilon_{n,k}$ its sign and by $|\nu_{n,k}|$ its positive normalization,
hence
\begin{equation}\label{positivenunk}
d|\nu_{n,k}|(\tau)=\frac{|h_{n,k}(\tau)|\,d|\sigma_{n,k}|(\tau)}{|P_{n,k-1}
(\tau)\,P_{n,k+1}(\tau)|}=\varepsilon_{n,k}\,\frac{h_{n,k}(\tau)\,d\sigma_{n,k}
(\tau)}{P_{n,k-1}(\tau)\,P_{n,k+1}(\tau)}.
\end{equation}

\begin{proposition}
For each $k=0,\ldots,p-1,$ the polynomial $p_{n,k}$ defined in
\eqref{eq:def:lpnk} satisfies the following:
\begin{align}
\int_{a_{k}}^{b_{k}} p_{n,k}(\tau)\,\tau^{s}\,d|\nu_{n,k}|(\tau) & =0,\qquad
s=0,\ldots,Z(n,k)-1,\label{orthog:lpnk}\\
\int_{a_{k}}^{b_{k}} p_{n,k}^{2}(\tau)\,d|\nu_{n,k}|(\tau) &
=1,\label{orthonorm:lpnk}
\end{align}
that is, $p_{n,k}$ is the orthonormal polynomial of degree $Z(n,k)$ with
respect to the positive measure $|\nu_{n,k}|$.

For each $k=1,\ldots,p$, the function $h_{n,k}$ defined in \eqref{eq:def:hnk}
satisfies
\begin{equation}\label{intrep:hnk}
h_{n,k}(z)=\begin{cases}
\varepsilon_{n,k-1}\int_{a_{k-1}}^{b_{k-1}}\frac{p_{n,k-1}^{2}(\tau)}{z-\tau}\,
d|\nu_{n,k-1}|(\tau) & \mathrm{if}\ k\leq \ell,\\[1em]
\varepsilon_{n,k-1}\,z\int_{a_{k-1}}^{b_{k-1}}\frac{p_{n,k-1}^{2}(\tau)}{z-\tau}
\,d|\nu_{n,k-1}|(\tau) & \mathrm{if}\ k>\ell.
\end{cases}
\end{equation}
\end{proposition}
\begin{proof}
The orthogonality conditions \eqref{orthog:lpnk} are obvious in view of
\eqref{eq:normvaryorthog:Pnk}. The formulas \eqref{orthonorm:lpnk} and
\eqref{intrep:hnk} follow immediately from
\eqref{eq:def:lpnk}--\eqref{eq:def:hnk}, \eqref{def:kappank},
\eqref{eq:def:Knm1}--\eqref{eq:def:Knk} and \eqref{intrep:Hnk}.
\end{proof}

\section{Zero asymptotic distribution}
\subsection{Definitions and results}\label{subsectionzerodef}

In this section we investigate the zero asymptotic distribution of the
polynomials $Q_{n}$. This distribution will be described in terms of a vector
equilibrium problem for logarithmic potentials. Before describing this problem,
let us introduce some definitions and notations.

Let $E_{k}$, $k=0,\ldots,p-1,$ be a system of compact subsets of the real line
satisfying
\begin{equation}\label{eq:condvep:1}
E_{k}\cap E_{k+1}=\emptyset,\qquad k=0,\ldots,p-2.
\end{equation}
We assume that
\begin{equation}\label{eq:condvep:2}
\mathrm{cap}(E_{k})>0,\qquad k=0,\ldots,p-1,
\end{equation}
where $\mathrm{cap}(E)$ denotes the logarithmic capacity of a compact set $E$.
A vector measure
\[
\vec{\nu}=(\nu_{0},\nu_{1},\ldots,\nu_{p-1})
\]
is called \emph{admissible} if
\begin{itemize}
\item[1)] $\nu_{k}$ is a positive Borel measure supported on $E_{k}$ for all
$k=0,\ldots,p-1$;
\item[2)] $\nu_{k}$ has total mass $\|\nu_{k}\|=1-\frac{k}{p}$ for all
$k=0,\ldots,p-1$.
\end{itemize}
We denote by $\mathcal{M}$ the class of all admissible vector measures.

Given a pair of compactly supported measures $\nu_1$, $\nu_{2}$, let
$I(\nu_{1})$ and $I(\nu_{1},\nu_{2})$ denote, respectively,
the logarithmic energy of $\nu_{1}$ and the mutual logarithmic energy of
$\nu_{1}$ and $\nu_{2}$ defined by
\[
I(\nu_{1})=\iint\log\frac{1}{|x-y|}\,d\nu_{1}(x)\,d\nu_{1}(y),\qquad
I(\nu_{1},\nu_{2})=\iint\log\frac{1}{|x-y|}\,d\nu_{1}(x)\,d\nu_{2}(y).
\]
On the class of admissible vector measures
$\vec{\nu}=(\nu_{0},\ldots,\nu_{p-1})$ we consider the energy functional $J$
defined by
\begin{equation}\label{def:energyfunc}
J(\vec{\nu}):=\sum_{k=0}^{p-1} I(\nu_{k})-\sum_{k=0}^{p-2}I(\nu_{k},\nu_{k+1}).
\end{equation}
Observe that $J$ is well-defined and $J(\vec{\nu})\in(-\infty,+\infty]$ for all
$\vec{\nu}\in\mathcal{M}$. This type of energy interaction is typical in the
study of Nikishin systems on the real line.

The vector equilibrium problem that is relevant in this work is the problem of
finding an extremal vector measure $\vec{\mu}\in\mathcal{M}$ that satisfies
\begin{equation}\label{equilprop}
J(\vec{\mu})=\inf_{\vec{\nu}\in\mathcal{M}} J(\vec{\nu})<\infty.
\end{equation}
Such a measure exists and is unique, see \cite{NikSor} for a proof of this fact
and several other important results on logarithmic vector equilibrium problems
in the complex plane. The extremal measure $\vec{\mu}$ is the vector
\emph{equilibrium measure}.

The vector equilibrium measure can be characterized in terms of certain
equilibrium conditions that we describe next. Given a vector measure
$\vec{\nu}=(\nu_{0},\ldots,\nu_{p-1})\in\mathcal{M}$, we consider the combined
potentials $W_{k}^{\vec{\nu}}$ defined by
\begin{equation}\label{def:combpot}
W_{k}^{\vec{\nu}}(z)=U^{\nu_{k}}(z)-\frac{1}{2} U^{\nu_{k-1}}(z)-\frac{1}{2}
U^{\nu_{k+1}}(z),\qquad k=0,\ldots,p-1,
\end{equation}
where $U^{\nu}$ denotes the logarithmic potential associated with $\nu$, i.e.,
\[
U^{\nu}(z)=\int\log\frac{1}{|z-t|}\,d\nu(t),
\]
and in \eqref{def:combpot} we understand $U^{\nu_{-1}}\equiv 0$,
$U^{\nu_{p}}\equiv 0$. The following result is an adaptation of a well-known
result in the theory of logarithmic vector equilibrium problems, see
\cite{NikSor}.

\begin{lemma}\label{lemma:varcond}
The measure $\vec{\mu}=(\mu_{0},\ldots,\mu_{p-1})\in\mathcal{M}$ is the 
equilibrium measure satisfying \eqref{equilprop} if and only if there exist finite
constants $\{w_{k}\}_{k=0}^{p-1}$ such that for every $k=0,\ldots,p-1$ the
following conditions hold:
\begin{align}
W_{k}^{\vec{\mu}}(x) & \leq w_{k},\qquad \mbox{for
all}\,\,x\in\supp(\mu_{k}),\label{eq:eqcond:1}\\
W_{k}^{\vec{\mu}}(x) & \geq  w_{k},\qquad \mbox{for q.e.}\,\,x\in
E_{k}.\label{eq:eqcond:2}
\end{align}
\end{lemma}

Let $E\subset\mathbb{C}$ be a compact set, let $\{\nu_{n}\}_{n}$ be a sequence
of finite positive measures supported on $E$, and let $\nu$ be another finite
positive measure on $E$. We write
\[
\nu_{n}\xrightarrow[n\rightarrow\infty]{*}\nu
\]
if for every $f\in C(E)$,
\[
\lim_{n\rightarrow\infty}\int f\, d\nu_{n}=\int f\, d\nu,
\]
i.e., when the sequence of measures converges to $\nu$ in the weak-star
topology. Given a polynomial $P$ of degree $n\geq 1$, we denote the associated
normalized zero counting measure by
\[
\mu_{P}=\frac{1}{n}\sum_{P(z)=0}\delta_{z},
\]
where $\delta_{z}$ is the Dirac mass at $z$ (in the sum the zeros are repeated
according to their multiplicity).

The weak asymptotic result that we present in this paper is obtained under mild
assumptions on the measures $\sigma$ generating the Nikishin system. One of
these assumptions is the so-called regularity of the measures in the sense of
Stahl and Totik. A measure $\sigma$ is said to be in the class \textbf{Reg} if
\begin{equation*}
\lim_{n\rightarrow\infty}\|\pi_{n}\|_{L^{2}(\sigma)}^{1/n}=\mathrm{cap}(\supp(\sigma)),
\end{equation*}
where $\pi_{n}$ denotes the $n$th monic orthogonal polynomial associated with
the measure $\sigma$.

We refer the reader to \cite{StahlTotik} for a detailed analysis of the
orthogonal polynomials associated with measures in the class \textbf{Reg}.
It is well-known that the regularity assumption is indeed a mild condition. For
instance, measures $\sigma$ supported on a compact interval
$I\subset\mathbb{R}$ on which $\sigma'(x)>0$ a.e. are regular.

Let $E\subset\mathbb{C}$ be a compact set with $\mathrm{cap}(E)>0$ and let
$\varphi$ be a continuous function on $E$. Recall that the equilibrium measure
$\overline{\mu}$ in the presence of the external field $\varphi$ is the unique
probability measure that minimizes the energy functional $I(\mu)+2\int\varphi
\,d\mu$ among all probability measures on $E$, cf. \cite{SaffTotik}. The
equibrium measure $\overline{\mu}$ satisfies
\begin{equation}\label{eqcond:extfield}
U^{\overline{\mu}}(z)+\varphi(z)\begin{cases}
\leq w, & \mbox{for all}\,\,z\in\supp(\overline{\mu}),\\
\geq w, & \mbox{for q.e.}\,\,z\in E,
\end{cases}
\end{equation}
for some constant $w$ (called the equilibrium constant). These equilibrium
conditions also characterize the equilibrium measure, and we emphasize that if
$E$ is regular with respect to the Dirichlet problem, then in
\eqref{eqcond:extfield} the first inequality can be replaced by an equality and
the second inequality holds for all $z\in E$.

We will need the following auxiliary result concerning the zero asymptotic
distribution of a sequence of orthogonal polynomials with respect to varying
measures.

\begin{lemma}\label{lemma:varyingmeas}
Let $\sigma\in\textbf{Reg}$, $E=\supp(\sigma)\subset\mathbb{R}$, where $E$ is
regular with respect to the Dirichlet problem. Let $\{\phi_{l}\}$,
$l\in\Lambda\subset\mathbb{Z}_{+}$, be a sequence of positive continuous
functions on $E$ such that
\[
\lim_{l\in\Lambda}\frac{1}{2l}\,\log\frac{1}{|\phi_{l}(x)|}=\varphi(x)>-\infty,
\]
uniformly on $E$. Let $q_{l}$, $l\in\Lambda$, be a sequence of monic
polynomials such that $\deg q_{l}=l$ and
\[
\int x^{k}\,q_{l}(x)\,\phi_{l}(x)\,d\sigma(x)=0,\qquad k=0,\ldots,l-1.
\]
Then
\[
\mu_{q_{l}}\xrightarrow[l\in\Lambda]{*}\overline{\mu}
\]
and
\[
\lim_{l\in\Lambda}\left(\int|q_{l}(x)|^{2}\,\phi_{l}(x)\,d\sigma(x)\right)^{1/2l
}=e^{-w},
\]
where $\overline{\mu}$ and $w$ are the equilibrium measure and equilibrium
constant in the presence of the external field $\varphi$ on $E$.
\end{lemma}

The above result was proved in \cite{FidLopLopSor}. It is a generalization of a
result of Gonchar and Rakhmanov \cite{GonRakh} obtained under the more
restrictive assumption that $\supp(\sigma)$ is an interval on which $\sigma'>0$
a.e.

In the following asymptotic results, the measures $\mu_k$ are the components of
the vector equilibrium measure  $\vec{\mu}=(\mu_{0},\ldots,\mu_{p-1})$ that
minimizes the energy functional \eqref{def:energyfunc} on
the space $\mathcal{M}$ of all admissible vector measures supported on
$E_{k}=\supp(\sigma_{k}^{*})$, $k=0,\ldots,p-1$, and the constants $w_k$  are
the equilibrium constants satisfying the variational
conditions \eqref{eq:eqcond:1}-\eqref{eq:eqcond:2}.

\begin{theorem}\label{theo:zeroasympPnk}
Let $(s_{0},\ldots,s_{p-1})=\mathcal{N}(\sigma_{0},\ldots,\sigma_{p-1})$ be the
Nikishin system generated by the measures $\sigma_{0},\ldots,\sigma_{p-1}$.
Assume that for each $k=0,\ldots,p-1$,
the measure $\sigma_{k}$ satisfies $\sigma_{k}\in\textbf{Reg}$ and
$\supp(\sigma_{k})$ is regular for the Dirichlet problem. Then, for each 
$k=0,\ldots,p-1$, we have $\supp(\mu_{k})=\supp(\sigma_{k}^{*})$, and the following three limit relations hold: 
\begin{equation}\label{zeroasympPnk}
\mu_{P_{n,k}}\xrightarrow[n\rightarrow\infty]{*}\frac{p}{p-k}\,\mu_{k},
\end{equation} 
\begin{align}\label{lim:coeffgeometricmean}
\lim_{m\to\infty}\left(\prod_{j=1}^m a_{pj+k}\right)^{1/m}=
e^{-\frac{2p}{p+1}\sum_{j=0}^k w_j},
\end{align}
and 
\begin{align}\label{asymptoticspsink}
\lim_{n\to\infty}|\psi_{n,k}(z)|^{1/Z(n,0)}=e^{-U^{\mu_{k}}(z)+U^{\mu_{k-1}}
(z)-2\sum_{j=0}^{k-1}w_j}
\end{align}
uniformly on compact subsets of $\mathbb{C}\setminus
\left([a_{k-1},b_{k-1}]\cup[a_{k},b_{k}]\cup \{0\}\right)$. In \eqref{asymptoticspsink} we understand $U^{\mu_{-1}}, U^{\mu_{p}}\equiv 0$, and this formula is also valid for $k=p$.
\end{theorem}

We now state the corresponding asymptotic results on the stars $\Gamma_k$. For
each $k=0,\ldots,p-1$, let $\tilde{\mu}_k$ be the unique rotationally symmetric
measure supported on $\Gamma_k$ such that for every Borel set $E\subset
[a_k,b_k]$,
\begin{equation}\label{def:muktilde}
\tilde{\mu}_k\left(\left\{z:z^{p+1}\in E\right\}\right)=\mu_k(E).
\end{equation}

Let $\omega_{k,j}$, $j=0,\ldots,p$, be the $p+1$ distinct  roots of the
equation $z^{p+1}=(-1)^k$, numbered as usual in such a way that $0\leq \arg
\omega_{k,j}<\arg\omega_{k,j+1}<2\pi$. Then we  can write
$\Gamma_k=\cup_{j=0}^{p}\Gamma_{k,j}$, with
\[
\Gamma_{k,j}=\left\{z:z^{p+1}\in [a_k,b_k],\ z/\omega_{k,j}\geq 0\right\}.
\]

Then, for every Borel set $F\subset \Gamma_{k,j}$,
\[
\left.\tilde{\mu}_k\right|_{\Gamma_{k,j}}(F)=\frac{1}{p+1}\mu_k\left(\{z^{p+1}
:z\in F\}\right).
\]
\begin{corollary}\label{Coro:zeroasympQn}
For the zero counting measures $\mu_{Q_n}$ of the multi-orthogonal polynomials
$Q_n$, we have
\begin{equation}\label{zeroasympQn}
\mu_{Q_{n}}\xrightarrow[n\rightarrow\infty]{*}\tilde{\mu}_{0},
\end{equation}
and for every $k=0,\ldots,p$,
\begin{align}\label{asymptoticsPsink}
\lim_{n\to\infty}|\Psi_{n,k}(z)|^{1/n}=e^{-U^{\tilde{\mu}_{k}}(z)+U^{\tilde{\mu}
_{k-1}}(z)-\frac{2}{p+1}\sum_{j=0}^{k-1} w_j}
\end{align}
uniformly on compact subsets of $\mathbb{C}\setminus \left(\Gamma_k\cup
\Gamma_{k-1}\cup\{0\}\right)$. In \eqref{asymptoticsPsink} we understand $U^{\tilde{\mu}_{-1}}, U^{\tilde{\mu}_{p}}\equiv 0$.
\end{corollary}

The proofs of these asymptotic results  make use of the following auxiliary lemma.
\begin{lemma}\label{lemma:regularity}
Let $\sigma_{j}$ be a positive, rotationally symmetric measure on the star
$\Gamma_{j}=\{z\in\mathbb{C}: z^{p+1}\in[a_{j},b_{j}]\}$, for some
$j=0,\ldots,p-1$, and suppose that $\sigma_{j}\in\textbf{Reg}$. Then the
measures $d\sigma_{j}^{*}(\tau)$ and $|\tau|\,d\sigma_{j}^{*}(\tau)$ on
$[a_{j}, b_{j}]$, where $\sigma_{j}^{*}$ is defined in \eqref{def:sigma:star},
are also in the class \textbf{Reg}.
\end{lemma}
\begin{proof}
We begin by observing that, since
 $
\supp(\sigma_j)=\{z:z^{p+1}\in \supp(\sigma^*_j)\}
$, we have \cite[Thm. 5.2.5]{Ransford}
\[
[\mathrm{cap}(\supp(\sigma_{j}))]^{p+1}=\mathrm{cap}(\supp(\sigma_{j}^{*})).
\]

Let $\pi_{n}$ be the $n$th monic orthogonal polynomial associated with the
measure $\sigma_{j}$. Then $\sigma_{j}\in\textbf{Reg}$ means that
\begin{equation}\label{regularitysigmaj}
\lim_{n\rightarrow\infty}\|\pi_{n}\|_{L_2(\sigma_j)}^{1/n}=\mathrm{cap}(\supp(\sigma_{j})).
\end{equation}
By the rotational symmetry of $\sigma_{j}$, the monic polynomial
$\omega^{-n}\pi_{n}(\omega z)$ (where $\omega=e^{\frac{2\pi i}{p+1}}$) is also orthogonal with respect to $\sigma_{j}$, and therefore, for every integer
$m\geq 0$, $
\pi_{m(p+1)}(z)=L_{m}(z^{p+1})
$, where $L_m$ is the $m$th monic orthogonal polynomial with respect to $\sigma_{j}^{*}$. Moreover, 
$\|\pi_{m(p+1)}\|_{L^{2}(\sigma_{j})}^{2}= \|L_m\|_{L^{2}
(\sigma_{j}^{*})}^{2}$, and so from \eqref{regularitysigmaj} it follows that
\[
\lim_{m\rightarrow\infty}\|L_m\|_{L^{2}(\sigma_{j}^{*})}^{1/m}
=[\mathrm{
cap}(\supp(\sigma_{j}))]^{p+1}=\mathrm{cap}(\supp(\sigma_{j}^{*})).
\]
This proves that $\sigma_{j}^{*}\in\textbf{Reg}$.

Now that we know that $\sigma_j^*$ is regular, we want to conclude that the
measure $d\lambda(\tau):=|\tau|\,d\sigma_{j}^{*}(\tau)$ is also regular. Let
$(l_{n})$ be the sequence of monic orthogonal polynomials
associated with $\lambda$, and let $(L_{n})$ be the corresponding
sequence for $\sigma_{j}^{*}$. 
Without loss of generality, we assume that $0\leq a_j\leq b_j$. Then, by the
extremality property of the monic orthogonal polynomials, we have
\[
b_j^{-1} \|L_{n+1}\|^2_{L^{2}(\sigma^*_{j})}\leq
b_j^{-1}\int_{a_j}^{b_j}|l_{n}(\tau)|^{2}\tau^2d\sigma^*_{j}(\tau)\leq
\|l_n\|^2_{L^{2}(\lambda)}\leq\int_{a_j}^{b_j}|L_{n}|^{2}d\lambda\leq  b_j
\|L_n\|^2_{L^{2}(\sigma^*_{j})}.
\]
Since 
$\supp(\sigma_{j}^{*})=\supp(\lambda)$, taking $n$th roots and letting $n\to\infty$
yields the desired result. 
\end{proof}

\subsection{Proof of Theorem~\ref{theo:zeroasympPnk}}

Let $\Lambda\subset\mathbb{N}$ be a sequence of integers such that for every
$k=0,\ldots,p-1$,
\begin{equation}\label{assumptionconv}
c_{k}\,\mu_{P_{n,k}}\xrightarrow[n\in\Lambda]{*}\mu_{k},\qquad
c_{k}:=1-\frac{k}{p},
\end{equation}
for some positive measures $\mu_{k}$ on $[a_{k},b_{k}]$. Our goal is to show
that the vector measure $\vec{\mu}=(\mu_{0},\ldots,\mu_{p-1})$ is the unique
equilibrium measure satisfying \eqref{equilprop}.
This implies that for each $k=0,\ldots,p-1,$ the sequence of measures
$(\mu_{P_{n,k}})_{n}$ has a unique limit point in the weak-star topology. By
the compactness of the unit ball in the space of Borel positive measures with
the weak-star topology, we obtain that the limits hold.

Note that $\mu_{k}$ has mass $c_{k}$. From \eqref{assumptionconv} we obtain
\begin{equation}\label{eq:logasympPnk}
\lim_{n\in\Lambda}\frac{c_{k}}{Z(n,k)}\,\log
|P_{n,k}(z)|=-U^{\mu_{k}}(z),\qquad k=0,\ldots,p-1,
\end{equation}
uniformly on compact subsets of $\mathbb{C}\setminus[a_{k},b_{k}]$. Note also that $\supp(\mu_{k})\subset E_{k}:=\supp(\sigma_{k}^{*})$ for every $k=0,\ldots,p-1$, which follows immediately from Corollary~\ref{cor:zeros}.

For each $k=0,\ldots,p-1$ we have the orthogonality conditions
\begin{equation}\label{varyingorthogPnk}
\int_{a_{k}}^{b_{k}}
P_{n,k}(\tau)\,\tau^{s}\,\frac{|h_{n,k}(\tau)|\,d|\sigma_{n,k}|(\tau)}{|P_{n,k-1
}(\tau)\,P_{n,k+1}(\tau)|}=0,\qquad s=0,\ldots,Z(n,k)-1,
\end{equation}
where $h_{n,k}$ is defined in \eqref{eq:def:hnk}, and recall that $P_{n,-1}, P_{n,p}\equiv 1$. We consider the
expression
\begin{equation}\label{varyingweightorthog}
\frac{1}{2
Z(n,k)}\log\left(\frac{|P_{n,k-1}(\tau)||P_{n,k+1}(\tau)|}{|h_{n,k}(\tau)|}
\right)
=\frac{\log |P_{n,k-1}(\tau)|+\log |P_{n,k+1}(\tau)|-\log
|h_{n,k}(\tau)|}{2Z(n,k)}
\end{equation}
associated with the orthogonality measure in \eqref{varyingorthogPnk}. Applying \eqref{eq:logasympPnk} and
\eqref{asympformZnk} we obtain
\begin{align}
\lim_{n\in\Lambda}\frac{\log |P_{n,k-1}(\tau)|}{2 Z(n,k)} &
=\lim_{n\in\Lambda}\frac{Z(n,k-1)}{Z(n,k)}\frac{\log |P_{n,k-1}(\tau)|}{2
Z(n,k-1)}=-\frac{p}{2(p-k)} U^{\mu_{k-1}}(\tau),\label{normlogasymp:1}\\
\lim_{n\in\Lambda}\frac{\log |P_{n,k+1}(\tau)|}{2 Z(n,k)} &
=\lim_{n\in\Lambda}\frac{Z(n,k+1)}{Z(n,k)}\frac{\log |P_{n,k+1}(\tau)|}{2
Z(n,k+1)}=-\frac{p}{2(p-k)} U^{\mu_{k+1}}(\tau),\label{normlogasymp:2}
\end{align}
uniformly on $E_{k}=\supp(\sigma_{k}^{*})$. To analyze the expression $\log |h_{n,k}(\tau)|/Z(n,k)$, observe that because $p_{n,k-1}^{2}|\nu_{n,k-1}|$ is a probability measure (cf. \eqref{orthonorm:lpnk}), we can find constants $d_{k}>0$,
$D_{k}>0$, independent of $n$, such that
\begin{equation}\label{est:hnk:2}
d_{k}\leq \int \frac{p_{n,k-1}^{2}(\tau)}{|z-\tau|}\,d|\nu_{n,k-1}|(\tau)\leq
D_{k}, \qquad z\in[a_{k},b_{k}].
\end{equation}
These bounds together with \eqref{intrep:hnk} yield
\begin{equation}\label{asympnormhnk}
\lim_{n\in\Lambda}\frac{\log |h_{n,k}(\tau)|}{2 Z(n,k)}=0,
\end{equation}
uniformly on $E_{k}=\supp(\sigma_{k}^{*})$. Here an observation needs to be
made concerning the posibility that $0\in E_{k}$. If this happens and
$\ell(n)<k$, then $|h_{n,k}(z)|$ has a factor $|z|$ (see \eqref{intrep:hnk}) which may destroy \eqref{asympnormhnk} at $\tau=0$. To avoid this, one could, if
necessary, take the limit as $n\rightarrow \infty$ along a subsequence
$\Lambda'\subset\Lambda$ such that $\ell(n)<k$ for all $n\in\Lambda'$ and
incorporate the factor $|z|$ to the measure $|\sigma_{n,k}|$ in
\eqref{varyingorthogPnk}. So we may assume that \eqref{asympnormhnk} holds in
any case.
As a result of \eqref{normlogasymp:1}, \eqref{normlogasymp:2} and
\eqref{asympnormhnk} we have the convergence of \eqref{varyingweightorthog} to
\[
\lim_{n\in\Lambda}\frac{1}{2
Z(n,k)}\log\left(\frac{|P_{n,k-1}(\tau)||P_{n,k+1}(\tau)|}{|h_{n,k}(\tau)|}
\right)=-\frac{p}{2(p-k)}(U^{\mu_{k-1}}(\tau)+U^{\mu_{k+1}}(\tau)).
\]
uniformly on $E_{k}$.

Now we can apply Lemma~\ref{lemma:varyingmeas} to the sequence $(P_{n,k})$, identifying $q_{l}$ with this sequence, $\phi_{l}$ with the weight $|h_{n,k}|/|P_{n,k-1}\,P_{n,k+1}|$, and $\varphi$ with
\begin{equation}\label{def:extfieldvarphi}
\varphi(z)=-\frac{p}{2(p-k)}(U^{\mu_{k-1}}(z)+U^{\mu_{k+1}}(z)).
\end{equation}
By Lemma~\ref{lemma:regularity}, the measures $|\sigma_{n,k}|$ are in the class $\textbf{Reg}$, and the regularity of $\supp(\sigma_{k})$ with respect to the Dirichlet problem also implies the regularity of $\supp(\sigma_{n,k})$ in the same sense. Hence by Lemma~\ref{lemma:varyingmeas} we get that for each $0\leq k\leq p-1$,
\begin{equation}\label{weakconvPnk}
\mu_{P_{n,k}}\xrightarrow[n\rightarrow\Lambda]{*}\overline{\mu}_{k}
\end{equation}
and
\begin{equation}\label{convergL2norm}
\lim_{n\in\Lambda}\left(\int_{a_{k}}^{b_{k}}
P_{n,k}^{2}(\tau)\,\frac{|h_{n,k}(\tau)|
d|\sigma_{n,k}|(\tau)}{|P_{n,k-1}(\tau)||P_{n,k+1}(\tau)|}\right)^{\frac{1}{2
Z(n,k)}}=e^{-\overline{w}_{k}},
\end{equation}
where $\overline{\mu}_{k}$ and $\overline{w}_{k}$ are the equilibrium measure and
equilibrium constant, respectively, in the presence of the external field \eqref{def:extfieldvarphi} on $E_{k}$. Hence, from \eqref{assumptionconv} and \eqref{weakconvPnk} it follows that for each $0\leq k\leq p-1$ we have $\mu_{k}=c_{k}\,\overline{\mu}_{k}$ and consequently
\begin{equation}\label{variational2}
U^{\mu_{k}}(x)-\frac{1}{2} U^{\mu_{k-1}}(x)-\frac{1}{2} U^{\mu_{k+1}}(x)
\begin{cases}
=c_{k}\,\overline{w}_{k}, & \mbox{for}\,\,x\in \supp(\mu_{k}), \\[0.5em]
\geq c_{k}\,\overline{w}_{k}, & \mbox{for}\,\,x\in E_{k}\setminus\supp(\mu_{k}),
\end{cases}
\end{equation}
cf. \eqref{eqcond:extfield}.

Finally, let $w_{k}:=c_{k}\overline{w}_{k}$, $k=0,\ldots,p-1$. Then \eqref{variational2} shows that the
vector measure $\vec{\mu}=(\mu_{0},\ldots,\mu_{p-1})\in\mathcal{M}$ satisfies
the variational conditions \eqref{eq:eqcond:1}-\eqref{eq:eqcond:2}
for every $k=0,\ldots,p-1$ (cf. \eqref{def:combpot}). Therefore, by Lemma
\ref{lemma:varcond}, $\vec{\mu}$ is the unique equilibrium
measure satisfying \eqref{equilprop}. This concludes the proof of \eqref{zeroasympPnk}.

From \eqref{variational2} and \eqref{ineqcombpot}
we deduce that $E_{k}\setminus\supp(\mu_{k})=\emptyset$ for all $k$,
hence $\supp(\mu_{k})=\supp(\sigma_{k}^{*})$ for all $k$.

We now prove \eqref{asymptoticspsink}. From \eqref{eq:def:Hnk} and \eqref{eq:def:hnk}, we see that
\begin{equation}\label{psinkandhnk}
|\psi_{n,k}(z)|=\frac{K^{-2}_{n,k-1}|h_{n,k}(z)||P_{n,k}(z)|\,}{|P_{n,k-1}(z)|},
\quad 1\leq k\leq p.
\end{equation}
Thus, we seek to establish the asymptotic behavior of each of the factors in
the right-hand side of
\eqref{psinkandhnk}.

Let us set
\[
f_{n,k}(z):=\begin{cases}
h_{n,k}(z), &\ell(n)< k,\\
zh_{n,k}(z),&\ell(n)\geq k.
\end{cases}
\]
From \eqref{orthonorm:lpnk}-\eqref{intrep:hnk}, we see that all functions in
the family $\{f_{n,k}\}_{n}$   are analytic in
$U:=\overline{\mathbb{C}}\setminus \left([a_{k-1},b_{k-1}]\cup\{0\}\right)$, and
for every closed subset $E$ of $U$, we have
$\sup_{n\geq 0}\max_{z\in E}|f_{n,k}(z)| <\infty$.
By Montel's theorem, $\{f_{n,k}\}_{n}$ is a normal family in $U$. Since no
$f_{n,k}$ vanishes in $U$, and particularly, $|f_{n,k}(\infty)|= 1$,  Hurwitz's
theorem tells us that every normal limit point of $\{f_{n,k}\}_{n}$ is
zero-free in $U$, which, in view of the normality of the family, implies that
for every closed subset $E\subset U$,
$\inf_{n\geq 0}\min_{z\in E}|f_{n,k}(z)| >0$. Therefore, as
$\lim_{n\to\infty}Z(n,0)=\infty$, we have
\begin{align}\label{asymptoticshnk}
\lim_{n\to\infty}|h_{n,k}(z)|^{1/Z(n,0)}=1
 \end{align}
locally uniformly on $\mathbb{C}\setminus
\left([a_{k-1},b_{k-1}]\cup\{0\}\right)$.

Applying \eqref{convergL2norm} and the equalities \eqref{eq:def:Knk} and \eqref{eq:def:hnk},  we  obtain
\begin{align*}
\lim_{n\to\infty}\left[\frac{K_{n,k}}{K_{n,k-1}}\right]^{1/Z(n,k)}=e^{w_k/c_{k}},
\quad k=0,\ldots,p-1.
 \end{align*}

Since $K_{n,-1}=1$ and $c_{0}=1$, this yields
 \[
 \lim_{n\to\infty}K_{n,0}^{1/Z(n,0)}=e^{w_0}.
\]
More generally,
\begin{align}\label{asymptoticsKnk}
 \lim_{n\to\infty}K_{n,k}^{1/Z(n,0)}=e^{\sum_{j=0}^k w_j},\qquad  k=0,\ldots,p-1,
 \end{align}
which easily follows by mathematical induction. Indeed,
\[
\lim_{n\to\infty}\frac{Z(n,k)}{Z(n,0)}=\lim_{n\to\infty}\prod_{j=1}^k\frac{Z(n,
j)}{Z(n,j-1)}=\prod_{j=1}^k\frac{p-j}{p-(j-1)}=\frac{p-k}{p}=c_k,\quad 0\leq
k\leq p-1,
\]
so that
\[
K_{n,k}^{1/Z(n,0)}=\left(\left[\frac{K_{n,k}}{K_{n,k-1}}\right]^{1/Z(n,k)}
\right)^{
\frac{Z(n,k)}{Z(n,0)}}K_{n,k-1}^{1/Z(n,0)}\xrightarrow[n\rightarrow\infty]{}
e^{\sum_{j=0}^{k} w_j}.
\]

In virtue of \eqref{zeroasympPnk} we have
\begin{align}\label{asymptoticsPnk}
\lim_{n\to\infty}|P_{n,k}(z)|^{1/Z(n,0)}=e^{-U^{\mu_{k}}(z)},\quad 0\leq k\leq
p-1,
\end{align}
locally uniformly on $\mathbb{C}\setminus [a_k,b_k]$. This last equality
already proves \eqref{asymptoticspsink} for the case $k=0$. For $1\leq k\leq
p$, the corresponding result follows from
\eqref{psinkandhnk}-\eqref{asymptoticsPnk}.

Finally, we prove \eqref{lim:coeffgeometricmean}. 
Since the coefficients $a_n$ are positive, it follows from
\eqref{twointegralswitha_n}, \eqref{repreformomentpsink}, and
\eqref{eq:def:Knk} that for all $n\geq p$, $n\equiv k \mod p$,
\begin{align*}
a_n={} & \frac{K_{n-p,k}^2}{K_{n,k}^2},
\end{align*}
so that
\begin{align*}
\prod_{j=1}^{m}a_{pj+k}=&\frac{K_{k,k}^2}{K_{mp+k,k}^2} ,\quad 0\leq k\leq p-1.
\end{align*}
Then, using \eqref{asymptoticsKnk} and \eqref{asympformZnk}, we obtain
\begin{align*}
\lim_{m\to\infty}\left(\prod_{j=1}^m
a_{pj+k}\right)^{1/m}=\lim_{m\to\infty}\left[\left(K_{k,k}^2/K_{mp+k,k}
^2\right)^{1/Z(mp+k,0)}\right]^\frac{Z(mp+k,0)}{m}=
e^{-\frac{2p}{p+1}\sum_{j=0}^k w_j}.
\end{align*}

\subsection{Proof of Corollary \ref{Coro:zeroasympQn}}

For each $0\leq j\leq p$, let $f_j:\Gamma_{0,j}\to [a_0,b_0]$ be the function
given by $f_j(t)=t^{p+1}$, which is clearly a homeomorphism. Since
$Q_n(z)=z^\ell P_{n,0}(z^{p+1})$, $Q_n$ has a zero at the origin of order
$\ell$,  and its remaining zeros are
the elements of the set  $\{f^{-1}_j(\tau): 0\leq j\leq p,\ P_{n,0}(\tau)=0\}$.
Thus, for every continuous function $F$ on $\Gamma_0\cup\{0\}$, we have
\begin{align*}
  \int Fd\mu_{Q_n}
  ={}&\frac{\ell
F(0)}{n}+\frac{n-\ell}{n(p+1)}\sum_{j=0}^p\frac{1}{\frac{n-\ell}{p+1}}\sum_{P_{n
,0}(\tau)=0} F(f^{-1}_j(\tau)),
\end{align*}
whence \eqref{zeroasympQn} easily follows. 

Since $\tilde{\mu}_k$ is rotationally symmetric and $\mu_k$ is the push forward
of $\tilde{\mu}_k$ by the map $z\to z^{p+1}$, we have
\begin{align}\label{relmus}
U^{\mu_k}(z^{p+1})=\sum_{j=0}^p\int_{\Gamma_k}\log\frac{1}{|z-t\omega^j|}d\tilde{\mu}
_k(t)=(p+1)U^{\tilde{\mu}_k}(z)
\end{align}
Then, \eqref{asymptoticsPsink} follows by combining \eqref{modifiedPsink}, \eqref{asympformZnk}, \eqref{asymptoticspsink}, and
\eqref{relmus}.

\section{Hermite-Pad\'{e} approximation}\label{sectionhermitepade}
\subsection{Definitions and results}\label{subsectionhermitedef}

In this section we study the Hermite-Pad\'{e} approximation to the system of functions
\begin{equation}\label{Cauchytransfsj}
\widehat{s}_{j}(z)=\int_{\Gamma_{0}}\frac{d s_{j}(t)}{z-t},\qquad 0\leq j\leq p-1,
\end{equation}
where $(s_{0},\ldots,s_{p-1})=\mathcal{N}(\sigma_{0},\ldots,\sigma_{p-1})$ is the
Nikishin system of measures defined in \eqref{def:sj}. For this, we follow closely the method
employed by Gonchar-Rakhmanov-Sorokin \cite{GonRakhSor} in their study of Hermite-Pad\'{e} approximants for
generalized Nikishin systems on the real line.

The problem of Hermite-Pad\'{e} approximation for the system of functions \eqref{Cauchytransfsj} is
the following. Given a multi-index $\vec{n}=(n_0, n_1,\ldots, n_{p-1})\in\mathbb{Z}_{+}^{p}$, we seek a non-zero polynomial $Q_{\vec{n}}$ with $\deg(Q_{\vec{n}})\leq |\vec{n}|:=n_{0}+\ldots+n_{p-1}$ such that for every $j=0,\ldots,p-1$,
\begin{equation}\label{def:HP}
Q_{\vec{n}}(z)\ \widehat{s}_{j}(z)-Q_{\vec{n},j}(z)=O\left(\frac{1}{z^{n_j+1}}\right),\qquad z\rightarrow\infty,
\end{equation}
where $Q_{\vec{n},j}$ is the polynomial part in the Laurent series expansion of $Q_{\vec{n}}\, \widehat{s}_{j}$ at $z=\infty$. It is easy to see that such a polynomial $Q_{\vec{n}}$ exists, since the $p$ conditions \eqref{def:HP} can be expressed equivalently as a homogeneous system of $|\vec{n}|$ linear equations with $|\vec{n}|+1$ unknowns (the coefficients of $Q_{\vec{n}}$), which always has a non-trivial solution. The vector of rational functions
\[
\left(\frac{Q_{\vec{n},0}}{Q_{\vec{n}}}, \frac{Q_{\vec{n},1}}{Q_{\vec{n}}}\ldots,\frac{Q_{\vec{n},p-1}}{Q_{\vec{n}}}\right)
\]
is called a Hermite-Pad\'e approximant associated with $\vec{n}$ for the system of functions \eqref{Cauchytransfsj}. If we integrate the expression $z^{l}(Q_{\vec{n}}(z) \widehat{s}_{j}(z)-Q_{\vec{n},j}(z))$, $l=0,\ldots,n_{j}-1,$ along a closed contour that surrounds $\Gamma_{0}$, it easily follows from \eqref{def:HP} and \eqref{Cauchytransfsj} that the polynomial $Q_{\vec{n}}$ satisfies the multi-orthogonality conditions
\[
\int_{\Gamma_{0}} Q_{\vec{n}}(z)\,z^{l}\,ds_{j}(z)=0,\qquad l=0,\ldots,n_{j}-1,\qquad 0\leq j\leq p-1.
\]

In this paper we will only consider Hermite-Pad\'e approximants associated with multi-indices $\vec{n}=(n_{0},\ldots,n_{p-1})\in\mathbb{Z}_{+}^{p}$ that are defined by
\begin{equation}\label{compindex}
n_{j}=\lf \frac{n-j-1}{p}\rf+1,\qquad j=0,\ldots,p-1,
\end{equation}
for a given integer $n\geq 0$. These multi-indices can be equivalently described as those satisfying the conditions $n_{0}\geq n_{1}\geq \cdots\geq n_{p-1}$ and $n_{p-1}\geq n_{0}-1$, and are uniquely determined by their norm $|\vec{n}|$, which equals $n$ if $\vec{n}$ is defined by \eqref{compindex}. Let $I$ denote the sequence of such multi-indices.

Given $\vec{n}\in I$ with $|\vec{n}|=n$, we see that $Q_{\vec{n}}$ satisfies \eqref{orthog:Qn}. Thus, if we assume, as we will do, that $Q_{\vec{n}}$ is monic, then by the normality of the Nikishin system we have that $Q_{\vec{n}}$ is unique and is given by $Q_{\vec{n}}=Q_{n}$. Moreover, since
\[
\int_{\Gamma_{0}}\frac{Q_{n}(t)}{z-t}\ d s_{j}(t)=O\left(\frac{1}{z^{n_{j}+1}}\right),\qquad z\rightarrow\infty,
\]
from \eqref{def:HP} and the identity
\begin{equation}\label{identityHP}
Q_{n}(z)\ \widehat{s}_{j}(z)-\int_{\Gamma_{0}}\frac{Q_{n}(z)-Q_{n}(t)}{z-t}\,d s_{j}(t)=\int\frac{Q_{n}(t)}{z-t}\ d s_{j}(t)
\end{equation}
it follows that the polynomials
\begin{equation}\label{def:Qnj}
Q_{n,j}(z):=\int_{\Gamma_{0}}\frac{Q_{n}(z)-Q_{n}(t)}{z-t}\,d s_{j}(t),\qquad j=0,\ldots,p-1,
\end{equation}
are the numerators of the Hermite-Pad\'{e} approximant associated with $\vec{n}$.

\begin{definition}
Given $n\geq 0$, for each $j=0,\ldots,p-1$ we define
\begin{align}
\Phi_{n,j+1}(z) & =\int\frac{Q_{n}(t)}{z-t}\ d s_{j}(t),\label{def:Phinj}\\
\delta_{n,j}(z) & =\frac{\Phi_{n,j+1}(z)}{Q_{n}(z)}.\label{def:deltanj}
\end{align}
\end{definition}

From \eqref{identityHP}-\eqref{def:deltanj} we deduce that
\begin{equation}\label{formremainder}
\delta_{n,j}=\widehat{s}_{j}-\frac{Q_{n,j}}{Q_{n}},\qquad j=0,\ldots,p-1,
\end{equation}
i.e., $\delta_{n,j}$ is the remainder in the approximation of $\widehat{s}_{j}$ by the $j$th component of the $n$th Hermite-Pad\'e approximant.

\begin{theorem}\label{theo:HP}
Under the same assumptions of Theorem~\ref{theo:zeroasympPnk}, we have that for every $j=0,\ldots,p-1$,
\begin{equation}\label{asympdeltanj}
\lim_{n\rightarrow\infty}|\delta_{n,j}(z)|^{1/n}=e^{-U^{\widetilde{\mu}_{1}}(z)+2 U^{\widetilde{\mu}_{0}}(z)-\frac{2}{p+1}w_{0}}
\end{equation}
uniformly on compact subsets of $\mathbb{C}\setminus(\bigcup_{i=0}^{j+1}\Gamma_{i}\cup\{0\})$, where the measures $\widetilde{\mu}_{0}$, $\widetilde{\mu}_{1}$ are defined in \eqref{def:muktilde}, and $w_{0}=\overline{w}_{0}$ is the equilibrium constant in \eqref{variational2}. In particular, for every $j=0,\ldots,p-1,$
\begin{equation}\label{eq:convHPapprox}
\lim_{n\rightarrow\infty}\frac{Q_{n,j}(z)}{Q_{n}(z)}=\widehat{s}_{j}(z),
\end{equation}
uniformly on compact subsets of $\mathbb{C}\setminus (\bigcup_{i=0}^{j+1}\Gamma_{i}\cup\{0\})$.
\end{theorem}

Before we give the proof of Theorem~\ref{theo:HP}, we make some remarks and prove an auxiliary result.

Note that $\Phi_{n,1}=\Psi_{n,1}$. More generally, one can show (see e.g. \cite[pg. 691]{GonRakhSor}) that for every $k=1,\ldots,p,$
\begin{equation}\label{relPhiPsi}
\Phi_{n,k}(z)=\sum_{i=1}^{k} (-1)^{i-1} \widehat{s}_{i,k-1}(z) \Psi_{n,i}(z),\qquad z\in\mathbb{C}\setminus\bigcup_{l=0}^{k-1}\Gamma_{l},
\end{equation}
where $\widehat{s}_{i,k-1}(z)$ denotes the Cauchy transform of the measure $s_{i,k-1}=\langle \sigma_{i},\ldots,\sigma_{k-1}\rangle$ (cf. \eqref{def:skj}), and we understand $\widehat{s}_{k,k-1}(z)\equiv 1$. Observe that \eqref{reductiontosegments} implies that the function $\widehat{s}_{i,k-1}(z)$ does not vanish on $\mathbb{C}\setminus(\Gamma_{i}\cup\{0\})$.

\begin{lemma}\label{lemma:potineq}
Let $(\mu_{0},\ldots,\mu_{p-1})$ be the vector equilibrium measure satisfying \eqref{zeroasympPnk}, and $w_{0},\ldots,w_{p-1}$ the associated equilibrium constants. Then, for every $k=0,\ldots,p-1$ we have
\begin{equation}\label{ineqcombpot}
2U^{\mu_{k}}(z)-U^{\mu_{k-1}}(z)-U^{\mu_{k+1}}(z)-2 w_{k}<0,\qquad\mbox{for all}\,\,\,z\in\overline{\mathbb{C}}\setminus\supp(\mu_{k}),
\end{equation}
where $U^{\mu_{-1}}\equiv 0$, $U^{\mu_{p}}\equiv 0$.
\end{lemma}
\begin{proof}
According to \eqref{variational2}, we have
\begin{equation}\label{variational3}
2 U^{\mu_{k}}(x)-U^{\mu_{k-1}}(x)-U^{\mu_{k+1}}(x)-2 w_{k}=0,\qquad x\in\supp(\mu_{k}),\qquad k=0,\ldots,p-1.
\end{equation}
This implies that $U^{\mu_{k}}$ is continuous on $\supp(\mu_{k})$, and hence $U^{\mu_{k}}$ is continuous on $\mathbb{C}$ for all $k$. The measure $2\mu_{k}-\mu_{k-1}-\mu_{k+1}$ has total mass $1+\frac{1}{p}$ if $k=0$ and has total mass $0$ for all other values of $k$. Therefore the function $2U^{\mu_{k}}(z)-U^{\mu_{k-1}}(z)-U^{\mu_{k+1}}(z)-2 w_{k}$ is subharmonic on $\overline{\mathbb{C}}\setminus\supp(\mu_{k})$. By the maximum principle for subharmonic functions applied to this function, \eqref{variational3} implies \eqref{ineqcombpot}.
\end{proof}

\subsection{Proof of Theorem~\ref{theo:HP}}

We already observed that $\Phi_{n,1}=\Psi_{n,1}$. In virtue of \eqref{zeroasympQn}, \eqref{asymptoticsPsink} and \eqref{def:deltanj}, we get
\[
\lim_{n\rightarrow\infty}|\delta_{n,0}(z)|^{1/n}=\lim_{n\rightarrow\infty}\frac{|\Psi_{n,1}(z)|^{1/n}}{|Q_{n}(z)|^{1/n}}=e^{-U^{\widetilde{\mu}_1}(z)+2 U^{\widetilde{\mu}_0}(z)-\frac{2}{p+1}w_0},
\]
uniformly on compact subsets of $\mathbb{C}\setminus(\Gamma_{1}\cup\Gamma_{0}\cup\{0\})$, which is \eqref{asympdeltanj} for $j=0$.

Let $1\leq j\leq p-1$. By \eqref{relPhiPsi} we can write for $z\in\mathbb{C}\setminus\left(\bigcup_{l=0}^{j}\Gamma_{l}\cup\{0\}\right)$,
\begin{equation}\label{relPhiPsi:2}
\Phi_{n,j+1}(z)=\widehat{s}_{1,j}(z)\Psi_{n,1}(z)\left(1-\frac{\widehat{s}_{2,j}(z)\Psi_{n,2}(z)}{\widehat{s}_{1,j}(z)\Psi_{n,1}(z)}+\cdots+(-1)^{j}\frac{\Psi_{n,j+1}(z)}{\widehat{s}_{1,j}(z)\Psi_{n,1}(z)}\right).
\end{equation}
Now, applying \eqref{relmus} and \eqref{ineqcombpot}, we see that for every $i=1,\ldots,j,$
\[
-U^{\widetilde{\mu}_{i}}(z)+U^{\widetilde{\mu}_{i-1}}(z)-\frac{2}{p+1}\sum_{l=0}^{i-1} w_{l}
>-U^{\widetilde{\mu}_{i+1}}(z)+U^{\widetilde{\mu}_{i}}(z)-\frac{2}{p+1}\sum_{l=0}^{i} w_{l},
\]
for all $z\in\mathbb{C}\setminus\supp(\widetilde{\mu}_{i})$. This implies, in virtue of \eqref{asymptoticsPsink}, that
\[
\lim_{n\rightarrow\infty}\frac{\Psi_{n,i+1}(z)}{\Psi_{n,1}(z)}=0, \qquad 1\leq i\leq j,
\]
locally uniformly on $\mathbb{C}\setminus\left(\bigcup_{l=0}^{j+1}\Gamma_{l}\cup\{0\}\right)$. This and \eqref{relPhiPsi:2} give
\[
\lim_{n\rightarrow\infty}|\delta_{n,j}(z)|^{1/n}
=\lim_{n\rightarrow\infty}\frac{|\Phi_{n,j+1}(z)|^{1/n}}{|Q_{n}(z)|^{1/n}}
=\lim_{n\rightarrow\infty}\frac{|\Psi_{n,1}(z)|^{1/n}}{|Q_{n}(z)|^{1/n}}=e^{-U^{\widetilde{\mu}_1}(z)+2 U^{\widetilde{\mu}_0}(z)-\frac{2}{p+1}w_0},
\]
uniformly on compact subsets of $\mathbb{C}\setminus\left(\bigcup_{l=0}^{j+1}\Gamma_{l}\cup\{0\}\right)$. From \eqref{variational3} for $k=0$ and \eqref{relmus} we see that $-U^{\widetilde{\mu}_1}(z)+2 U^{\widetilde{\mu}_0}(z)-\frac{2}{p+1}w_0<0$ on $\mathbb{C}\setminus\left(\bigcup_{l=0}^{j+1}\Gamma_{l}\cup\{0\}\right)$, hence \eqref{eq:convHPapprox} follows from \eqref{def:deltanj} and \eqref{asympdeltanj}.

\section{Acknowledgements}

We thank Ulises Fidalgo-Prieto for many useful discussions. The first author was partially supported by the grant MTM2015-65888-C4-2-P of the Spanish Ministry of Economy and Competitiveness.

\end{document}